\documentclass[11 pt]{amsart}

\usepackage[T1]{fontenc}

\title[Spaces of Besov-Sobolev type] {Spaces of Besov-Sobolev  type and a problem on nonlinear approximation }

\author[ \'O. Dom\'inguez \    A. Seeger \   B. Street \  J. Van Schaftingen \  P.-L. Yung]
{  \'Oscar Dom\'inguez \ \ \ \ Andreas Seeger \ \ \ \ Brian Street \\ \\  Jean Van Schaftingen \ \ \ \   Po-Lam Yung}

 \address{\'Oscar Dom\'inguez,  Universit\'e Lyon 1, Institut Camille Jordan, 43 Blvd du 11 novembre 1918, F-69622 Villeurbanne cedex, France; Departamento de An\'alisis Matem\'atico y Matem\'atica Aplicada\\  Faculdad de Matem\'aticas\\ Universidad Complutense de Madrid \\Plaza de Ciencias 3\\28040 Madrid\\Spain.}
\email{oscar.dominguez@ucm.es; dominguez@math.univ-lyon1.fr}
\address{Andreas Seeger, Department of Mathematics \\ 
University of Wisconsin, Madison \\
480 Lincoln Drive, Madison, WI, 53706, USA}
\email{seeger@math.wisc.edu}

\address{Brian Street, Department of Mathematics \\ 
University of Wisconsin, Madison \\
480 Lincoln Drive, Madison, WI, 53706, USA}
\email{street@math.wisc.edu}

\address{Jean Van Schaftingen, Universit\'e catholique de Louvain\\ 
Institut de Recherche en Math\'ematique et Physique\\
Chemin du Cyclotron 2 bte L7.01.01\\
1348 Louvain-la-Neuve\\
Belgium}
\email{Jean.VanSchaftingen@UCLouvain.be}

\address{Po-Lam Yung, Mathematical Sciences Institute \\
Australian National University \\
Canberra ACT 2601 \\
Australia} 
\email{PoLam.Yung@anu.edu.au}

\usepackage{amssymb}
\usepackage{amsmath}
\usepackage{mathtools}
\usepackage{amsfonts}
\usepackage{graphicx}
\usepackage{amsthm}
\usepackage[shortlabels]{enumitem}
\usepackage{mleftright}
\usepackage{microtype}
\usepackage[mathscr]{eucal}
\usepackage{verbatim}
\usepackage{color}
\usepackage{hyperref}
\usepackage{tikz-cd}
\usepackage{esint}

 \theoremstyle{plain}
\newtheorem{theorem}{Theorem}[section]
\newtheorem{lemma}[theorem]{Lemma}
\newtheorem{proposition}[theorem]{Proposition}
\newtheorem{corollary}[theorem]{Corollary}
\newtheorem{defn}[theorem]{Definition}
\numberwithin{equation}{section}
\theoremstyle{plain}

\numberwithin{equation}{section}
\theoremstyle{remark}

\newtheorem{remark}[theorem]{Remark}

\newtheorem*{remarka}{Remark}

\theoremstyle{plain}

\usepackage{bbm}
\def\bbone{{\mathbbm 1}}

\definecolor{ascol}{rgb}{0,0,1.} 
\definecolor{jvscol}{cmyk}{.74, 0, 1, .41} 
\definecolor{pycol}{rgb}{0.9,.5,0} 

\newcommand{\py}[1]{{\color{blue}{{[P: #1]}}}}

\newcommand{\Be}{\begin{equation}}
\newcommand{\Ee}{\end{equation}}
\newcommand{\EMap}{\mathscr{E}_M}
\newcommand{\cSt}{\widetilde{\cS}}
\newcommand{\Biginclude}{\mathcal{I}}
\newcommand{\JMap}{\mathscr{J}_M}
\newcommand{\Ball}[2]{B^{d}(#1,#2)}
\newcommand{\Indicator}{\bbone}

\def\high{{\mathrm{high}}}
\def\low{{\mathrm{low}}}

\def\intslash{\rlap{\kern  .32em $\mspace {.5mu}\backslash$ }\int}
\def\qsl{{\rlap{\kern  .32em $\mspace {.5mu}\backslash$ }\int_{Q_x}}}

\def\vth{\vartheta}

\def\cf{\emph{cf.\ }}

\def\loc{{\text{\rm loc}}}

\def\supp{{\text{\rm supp }}}

\def\inn#1#2{\langle#1,#2\rangle}

\def\meas{{\text{\rm meas}}}

\def\lc{\lesssim}
\def\gc{\gtrsim}

\def\ga{\gamma}

\def\eps{\varepsilon}
\def\ep{\epsilon}

\def\ka{\kappa}

\def\la{\lambda}

\def\Om{\Omega}

\def\fB{{\mathfrak {B}}}

\def\fN{{\mathfrak {N}}}

\def\bbC{{\mathbb {C}}}
\def\bbD{{\mathbb {D}}}

\def\bbN{{\mathbb {N}}}

\def\bbR{{\mathbb {R}}}

\def\bbZ{{\mathbb {Z}}}
\def\cA{{\mathcal {A}}}
\def\cB{{\mathcal {B}}}

\def\cI{{\mathcal {I}}}

\def\cK{{\mathcal {K}}}
\def\cL{{\mathcal {L}}}
\def\cM{{\mathcal {M}}}

\def\cP{{\mathcal {P}}}
\def\cQ{{\mathcal {Q}}}

\def\cS{{\mathcal {S}}}
\def\cT{{\mathcal {T}}}

\def\cZ{{\mathcal {Z}}}

\newcommand{\R}{\mathbb{R}}
\newcommand{\Z}{\mathbb{Z}}

\newcommand{\N}{\mathbb{N}}

\newcommand{\qu}{\mathscr D}
\newcommand{\poi}{\mathscr P}
\newcommand{\napoi}{\mathscr K}

\newcommand{\dif}{\,\mathrm{d}}

\DeclarePairedDelimiter{\abs}{\lvert}{\rvert}
\DeclarePairedDelimiter{\Norm}{\lVert}{\rVert}

\DeclarePairedDelimiterX\set[1]\{\}{%

#1
}

\newtheoremstyle{named}{}{}{\itshape}{}{\bfseries}{.}{.5em}{#1 \thmnote{#3}}
\theoremstyle{named}

\begin{document}

\begin{abstract} We study fractional variants of the quasi-norms introduced by  Brezis, Van Schaftingen, and Yung in the study of the Sobolev space $\dot W^{1,p}$.   The resulting spaces are  identified as  a special class of real  interpolation spaces of Sobolev-Slobodecki\u{\i}  spaces.
We establish the  equivalence between Fourier analytic definitions  and  definitions  via difference operators acting on measurable functions.  We  prove various  new results on  embeddings and non-embeddings, and give   applications to  harmonic and caloric
extensions. For suitable  wavelet bases  we obtain a characterization of the approximation spaces for best $n$-term approximation from a wavelet basis via smoothness conditions on the function; this extends   a classical result by DeVore, Jawerth and Popov. 

\end{abstract}

\maketitle

\section{Introduction and statements of results}
For $d \ge 1$, $b \in \R$ and a locally integrable function $f\in L^1_\loc(\bbR^d)$ consider the difference quotient 

\Be \label{eq:Dbdef} \qu_{b} f(x,y)= \frac{f(x)-f(y)}{|x-y|^{b}}, \qquad (x,y) \in \R^d \times \R^d = \R^{2d}.\Ee
Ha\"im Brezis and two of the authors \cite{br-sch-yung1} discovered that for  $f \in C^{\infty}_c(\bbR^d)$ and $1 \le p < \infty$, the Marcinkiewicz quasi-norm $[\qu_{1+d/p} f]_{L^{p,\infty}(\bbR^{2d})} $
is comparable to the Gagliardo-seminorm $\|\nabla f\|_{L^p(\bbR^d)}$ (see also \cite{nguyen06}, \cite{bsvy} for related results).  Using this equivalence, they considered in  \cite{br-sch-yung-lor}  certain borderline Gagliardo-Nirenberg interpolation inequalities that fail, and proved substitutes such as $[\qu_{s+d/p}  f]_{L^{p,\infty}(\bbR^{2d}) } \lesssim \|f\|_{L^{\infty}(\R^d)}^{1-s} \|\nabla f\|_{L^1(\R^d)}^{s}$ for $s = 1/p$ and $1 < p < \infty$, raising the natural question of what can be said about the  class of functions for which 
$[\qu_{s+d/p}  f]_{L^{p,\infty}(\bbR^{2d}) }$ is finite for $0<s<1$. 
This class was also considered in the papers by Poliakovsky \cite{poliakovsky} who asked about a more specific relation to Besov spaces, and   in the work by Dom\'inguez 
and Milman \cite{dominguez-milman} who considered abstract versions of \cite{br-sch-yung1}. 
As a  special case of our main results we show that
the above fractional variant  arises as a real interpolation space of a family of  homogeneous  Sobolev-Slobodecki\u\i\,  spaces $\dot W^{s,p}$.
 Henceforth, for $0 < s < 1$ and $1 < p < \infty$, the space $\dot W^{s,p}$ consists of all equivalent classes of measurable, finite a.e. functions $f$ (modulo equality a.e. and additive constants) for which $\qu_{s+{d}/{p}}f \in L^p(\R^{2d})$, with semi-norm $\|f\|_{\dot W^{s,p} } = \|\qu_{s+d/p}  f\|_{L^p(\bbR^{2d})}$; this space can be naturally identified the diagonal Besov space $\dot B^s_{p,p}$ (see e.g. the case $r=p$ in Theorem~\ref{differences-thm-intro} below).
We will show that for 
$p_0,p_1\in (1,\infty)$ such that $p_0<p<p_1$ and  $0<s+\frac dp-\frac{d}{p_i}<1$ the norm on the interpolation space
$[ \dot W^{s+\frac dp-\frac d{p_0},p_0}, \dot W^{s+\frac dp-\frac d{p_1}, p_1} ]_{\theta,\infty}$ is equivalent with the quasi-norm
$\|\qu_{s+\frac dp} f\big \|_{L^{p,\infty}(\bbR^{2d})}$.

The class of functions for which $\|\qu_{s+d/p} f \|_{L^{p,\infty}(\bbR^{2d})} $ 
is finite was labelled $BSY^s_p$  in \cite{dominguez-milman}. Here we  shall denote it by $\dot\cB^s_p(d,\infty) $ as 
 it  will arise as a member of a  natural and more general scale of spaces $\dot \cB^s_{p}(\gamma,r)$. 
 We begin by  giving  a Fourier analytic definition of the spaces $\dot \cB^s_p(\ga,r)$, which extends the classical definition of the homogeneous Besov space $\dot B^s_{p,p}$; in fact $\dot \cB^s_p(\ga,r)$  all coincide with $\dot B^s_{p,p}$ when $r=p$ (regardless of the value of~$\ga$).  We have  learned in the final stage of preparation of this paper  that V.L. Krepkogorski\u\i{} had already  introduced  the inhomogeneous variants of these classes in a little noticed  paper \cite{krepkogorskii} in 1994 and proved that they occur as interpolation spaces for Sobolev and other spaces; see  Remark  \ref{rem:inhomogeneous spaces} and the comments before Theorem \ref{interpol-thm-Fourier} below.


\subsection*{Variants of \texorpdfstring{Besov-Sobolev}{Besov-Sobolev} spaces}
We let $\varphi\in C^\infty_c({\bbR}^d)$ be a radial function with 
\begin{subequations}\label{eq:varphiassumptions}
\begin{align}
    &\mathrm{supp}(\varphi) \subset \{\xi: 3/4<|\xi|<7/4\},
    \\
    &\varphi( \xi)=1 \text{  for } 7/8\le |\xi|\le 9/8,
    \\
    &\sum_{k\in \bbZ} \varphi(2^{-k}\xi)=1
    \text { for all }\xi\neq 0.
\end{align}
\end{subequations}
It is easy to check that the three requirements can be achieved.
 For a tempered distribution $f$ we define  the frequency localizations $L_k f$ via the Fourier transform by \[\widehat {L_kf}(\xi)=  \varphi(2^{-k}\xi)\widehat f(\xi). \]
We recall 
 the definition of the diagonal  homogeneous Besov spaces $\dot B^s_{p,p}$. Consider the space $\cS_\infty (\bbR^d)$ of Schwartz functions whose Fourier transforms vanish to infinite order at $0$; this space carries the natural Fr\'echet topology inherited from the space of Schwartz functions. 
 We let $\cS_\infty'(\bbR^d)$ denote the dual space; it can be identified with  the 
 space  of tempered distributions modulo polynomials. The space $\dot B^s_{p,p}$ is defined as the subspace of $f\in \cS_\infty'(\bbR^d)$ 
for which \[\|f\|_{\dot B^s_{p,p}} \coloneqq\Big(\sum_{k\in \bbZ} \int_{\R^d} \big |2^{ks}L_k f(x) \big|^p \dif  x \Big)^{1/p} \] is finite.

  We will now define various  Lorentz versions of these spaces where a Lorentz norm is taken on the space $\bbR^d\times \bbZ$.
Recall that if $(\Omega,\mu) $ is a measure space and  $0<p,r<\infty$, the Lorentz space $L^{p,r}(\Omega,\mu)$ is defined as the space of measurable functions $g$ on $\Omega$ for which 
\[
[g]_{L^{p,r} (\Omega,\mu)}= \Big(r\int_0^\infty \la^r\mu(\{x\in \Omega:|g(x)|>\la\})^{r/p}\frac{\dif \la}{\la} \Big)^{1/r}
\] is finite. For $r=\infty$  we set $[g]_{L^{p,\infty}(\Om,\mu)} =\sup_{\la>0} \la \mu(\{|g|>\la\})^{1/p}$. The space  $L^{p,r}$ is normable when $1<p<\infty$, $1\le r\le \infty$, and 
for simplicity we will only consider these parameter ranges.
The precise expression for the norm is not important for  this paper; a suitable choice (\cite{hunt}) is   \[\|g\|_{L^{p,r}} = \Big(\int_0^\infty [t^{1/p} g^{**}(t) ]^r\frac{\dif t}{t}\Big)^{1/p}\]
where  $g^{**}(t)=t^{-1}\int_0^t g^*(s) \dif s$
and $g^{*}$ denotes  the nonincreasing rearrangement of $g$.  We now give a differentiation  of the spaces $\dot B^s_p(\gamma,r)$  {\it for all  $s\in \bbR$}, which matches the usual definition in, say \cite{BL, triebel} for the case $p=r$. Following the definition we will then formulate  characterizations for positive $s$ which link the general definition to the expressions involving the generalized difference quotients in \eqref{eq:Dbdef}.

\begin{defn} \label{first-definition}
 Let $\ga \in \bbR$. 
 
 (i) For 
a  measurable subset  $E$ of $\bbR^d\times \bbZ$  let
$\bbone_E$ be the indicator function of $E$ and 
\[\mu_\ga(E) = \sum_{k\in \bbZ}  {2^{-k\ga}}  \int_{\R^d}  \bbone_{E}(x,k)    \dif x.\]

(ii)
For $b\in \bbR$ define $P^b f: \bbR^d\times \bbZ\to \bbC$ by \[P^b f(x,k)=2^{k b } L_k f(x).\] 

(iii) For $s \in \bbR$, $1<p<\infty$, $1\le r\le\infty$, 
 let  $\dot \cB^s_p (\ga,r) $ be the space of $f\in \cS_\infty'(\bbR^d)$ such that
 the  function $P^{s+\frac \ga p}f$ belongs to the Lorentz space $ L^{p,r}(\bbR^d\times\bbZ; \mu_{\ga })$  and let 
 \Be\label{Bspgar-norm}\|f\|_{\dot \cB^s_p(\ga ,r)}=  \big\|P^{s+\frac \ga p} f\big\|_{L^{p,r} (\mu_{\ga})}.
\Ee
\end{defn}

Unravelling the definition, with $\meas \,A$ denoting the Lebesgue measure of $A\subset \bbR^d$, if $1 \leq r < \infty$ we get the following equivalence
\begin{multline} \label{eq:concretedef}
\|f\|_{\dot \cB^s_p(\ga,r)}\approx \\
\Big(r\int_0^\infty \la^r\Big[\sum_{k\in\bbZ } 2^{-k\ga}
\meas\big\{ x\in \bbR^d: |L_k f(x)|>\la 2^{-k(s+\frac\ga p)} \big\}\Big]^{r/p} 
\frac{\dif \la}{\la}\Big)^{1/r}
\end{multline}
whereas
\begin{multline} \label{eq:concretedef2}
\|f\|_{\dot \cB^s_p(\ga,\infty)}\approx \sup_{\lambda > 0} \la\Big[\sum_{k\in\bbZ } 2^{-k\ga}
\meas\big\{ x\in \bbR^d: |L_k f(x)|>\la 2^{-k(s+\frac\ga p)} \big\}\Big]^{1/p}.
\end{multline}
It is easy to check that we always have $\cS_{\infty}(\R^d) \subseteq \dot \cB^s_p(\ga,r)$. 
Note that  a simple Fubini-type argument gives
\Be\label{independence-of-gamma} \dot \cB^s_p (\ga,p)= \dot B^s_{p,p}, \text{  for all $\gamma\in \bbR$.}
\Ee
In contrast,  for $r\neq p$ the spaces $\dot \cB^s_p(\ga,r)$ depend on $\ga$ (see  Theorem \ref{thm:non-emb} (ii) below). 


\begin{remark}[Inhomogeneous versions] \label{rem:inhomogeneous spaces}We may also consider inhomogeneous versions of the above spaces. Define   
\Be\label{eq:inhLP}
\text{\L}_k=L_k \text{  for } k>0,  \qquad  \text{\L}_0:=\mathrm{Id}-\sum_{k>0} \text{\L}_k. \Ee
 For $E\subset \bbR^d\times\bbN_0$ let  
$\widetilde  \mu_{\gamma}(E) =\sum_{k=0}^\infty 2^{-k\gamma} \int\bbone_E(x,k) \dif x$. Define $\varPi^b f(x,k)=2^{kb}\text{\L}_k f(x)$ for $k=0,1,2,\dotsc$  We may then define $\cB^s_p(\ga,r) $ to be the space of all tempered distributions $f \in \cS'(\R^d)$ such that \Be\label{eq:inh-def}\|f\|_{\cB^s_p(\ga,r)} \coloneqq \|\varPi^{s+\gamma/p} f\|_{ L^{p,r}(\bbR^d\times \bbN_0,\widetilde \mu_\gamma)}\Ee
is finite. These spaces have already been defined by Krepkogorski\u \i {}
 \cite{krepkogorskii}, who  used the notation $BL^{s,k}_{p,q}$. The space  $\cB^s_p(\gamma,r)$ corresponds to  $BL^{s,-\gamma}_{p,r} $   in the notation  of \cite{krepkogorskii}.
\end{remark} 

\subsection*{Characterizations via difference operators} In order to explore the relation  to  the characterization of Sobolev spaces via weak-type quasinorms for  difference operators used in \cite{br-sch-yung1, bsvy} we seek equivalent definitions of  the spaces $\dot\cB^s_p(\ga,r)$ to spaces defined via  difference operators, at least for $s>0$. 
Let \[\Delta_h f(x)= f(x+h)-f(x),\] and define for $M\ge 2$ inductively $\Delta_h^M=\Delta_h\Delta_h^{M-1}$. These operations extend to tempered distributions.
We define a measure  $\nu_\ga $ on Lebesgue measurable subsets of $\bbR^d\times (\bbR^d\setminus \{0\})$   by
\[\nu_\gamma(E) = \iint_E \dif x \frac{\dif h}{|h|^{d-\gamma}}.\]
Also  define, for any Lebesgue measurable $f$ 
and 
$h\neq 0$, 
\[\cQ_{M,b}  f(x,h)= \frac{\Delta_h^M f(x)} {|h|^{b }}. \]
We denote by $\cT$ the   space  of tempered  functions; here  $g\in \cS'$  is a {\it tempered function}  if $g\in L^1_\loc(\bbR^d)$ and  if there exists  an $N<\infty$ such that \Be\label{eq:tempered} \int_{\R^d} |g(x)|(1+|x|)^{-N} \dif x<\infty. 
\Ee

Let $\cP_{M-1}$ denote  the set of polynomials of degree less than $M$. We wish to characterize $\dot \cB^s_p(\ga,r)$ in terms of the operators  $\cQ_{M,b}$ which annihilate $\cP_{M-1}$.  As 
$\dot \cB^s_p(\ga,r) \subset \cS_{\infty}'$,  every element $f \in \dot \cB^s_p(\ga,r)$ is actually an equivalent class $[f]$ of tempered distributions modulo {\it all} polynomials. Using the following theorem, if $0 < s < M$ and $M \in \N$, we  determine, for each $f \in \dot \cB^s_p(\ga,r)$, a subset of $[f]$, so that all elements of this subset differ by a polynomial in $\cP_{M-1}$.  Each element of this subset will be called a \emph{representative} of $f$ modulo $\cP_{M-1}$. This is often useful in practice, because then it makes sense to define, for example, any derivative of $f$ of order $\geq M$, and to define the convolution of $f$ with any Schwartz function that has $M$ vanishing moments.   For the classical Besov and Triebel-Lizorkin spaces (in particular $\dot B^s_{p,p} $) this is already addressed in Bourdaud's theory of  {\it realized spaces} \cite{bourdaud-besov}, in fact for $\dot\cB^s_p(\ga,p)\equiv \dot B^s_{p,p}$ one  part of the theorem is subsumed in \cite{bourdaud-besov}.

\begin{theorem} \label{differences-thm-intro}
Let $0<s<M$, $1<p<\infty$,  \(1\le r\leq \infty\) and \(\gamma\in \R\).  There exist positive constants $C_1$, $C_2$ so that the following holds.

\begin{enumerate}[(i)]
\item Let $f\in \dot \cB^s_p(\ga ,r)$. Then there exists  
a tempered  function $f_\circ$
such that 
\begin{equation} \label{eq:f_circ_def}
\inn{f}{\phi} =\int_{\R^d} f_\circ(x) \phi(x) \dif x \text{ for all $\phi\in \cS_\infty$ }
\end{equation} 
and 
 \begin{equation} \label{eq:Thm1.2i} \|\cQ_{M,  s+\frac \ga p} f_\circ
\|_{L^{p,r}( \nu_\ga)}
\le C_1 
\|f\|_{\dot\cB^s_p(\ga,r)}.
\end{equation}
The a.e.\ equivalent class of the function $f_{\circ}$ is unique modulo $\cP_{M-1}$; we refer to  the function  $f_{\circ}$ as  a \emph{representative} of $f$ modulo $\cP_{M-1}$. 

\item\label{Item::DiffThm::Meas} 

Suppose  $f: \bbR^d\to \bbC$ is a measurable\footnote{A main novelty
of Theorem \ref{differences-thm-intro} is that $f$ is merely assumed to be measurable in \ref{Item::DiffThm::Meas}.
In previous works on homogeneous Besov spaces   there is the additional a priori assumption $f\in L^1_\loc$. 
One way in which our result differs 
is that we show this   assumption is superfluous: 
the function \(f\) in Theorem \ref{differences-thm-intro} \ref{Item::DiffThm::Meas}  
is a priori only assumed to be measurable and we conclude that it is locally
integrable. }
function
satisfying \[\cQ_{M,s+\frac \ga p} f\in L^{p,r} (\nu_\ga).\] Then $f$ is  a tempered function, and 
under the natural identification in $\cS_\infty'$,   we have $f \in \dot \cB^s_p(\ga,r)$ with 
\[\|f\|_{\dot\cB^s_p(\ga ,r)}\le C_2 \|\cQ_{M, s+\frac \ga p} f
\|_{L^{p,r}( \nu_\ga)}.\]

\end{enumerate}
\end{theorem}

% ... 


Theorem \ref{differences-thm-intro} will be proved in \S\ref{sec:Proof-of-equiv},
where a more abstract equivalent statement is also given (Theorem \ref{Thm::Diff::MainDiffThm}).

\subsection*{Embeddings and non-embeddings}  We  establish  various  embedding relations  which sharpen previous results. We relate our classes to standard homogeneous Besov and Triebel-Lizorkin spaces and their Lorentz-space counterparts $\dot B^s_q[L^{p,r}]$ and $\dot F^{s}_{q}[L^{p,r}]$. These are defined as the  subspaces of $f\in \cS_\infty'(\bbR^d)$  for which
\begin{align}\label{eq:Besov-Lorentz}
\|f\|_{\dot B^{s}_q[L^{p,r}]}
&= \Big(\sum_{k\in \bbZ} 2^{ksq}\|L_kf\|_{L^{p,r}(\R^d)}^q\Big)^{1/q}
\\ \label{eq:TL-Lorentz}
\|f\|_{\dot F^{s}_q[L^{p,r}]}
&= \Big\|\Big(\sum_{k\in \bbZ} 2^{ksq}|L_kf|^q\Big)^{1/q}\Big\|_{L^{p,r}}
\end{align} 
are finite, respectively. The inhomogeneous analogues $B^s_q[L^{p,r}]$, $F^s_q[L^{p,r}]$ are defined analogously using the frequency localizations $\text{\L}_k$, $k\ge 0$ in \eqref{eq:inhLP}. 

For the standard Besov and Triebel-Lizorkin spaces one  works with the underlying $L^p$ metric, i.e.\ they are recovered by setting $r=p$ and we have
$\dot B^s_{p,q} = \dot B^s_q[L^p]$,  and $\dot F^s_{p,q} = \dot F^s_q[L^p]
$. For embedding relations among them one may consult   \cite{SeTr}  (however some care is needed since  the results in \cite{SeTr} are formulated for the inhomogeneous versions $B^{s}_q[L^{p,r}] $, $F^{s}_q[L^{p,r}]$).

\begin{theorem} \label{thm:embeddings} The following statements hold for all $s\in \bbR$, $p\in [1,\infty)$.

(i) For all 
$\ga \in \bbR$,
\begin{align*} &\dot \cB^s_p(\ga ,r)\hookrightarrow \dot B^s_r[L^{p,r}],\quad  p\le r \le \infty,
\\
&\dot B^s_r[L^{p,r}] 
\hookrightarrow   \dot \cB^s_p(\ga ,r),
\quad  1\le r\le p.
\end{align*}


(ii) Let $\ga \neq 0$. Then, 
\begin{align*} 
&\dot  F^s_{p,r} \hookrightarrow  \dot   \cB^s_p(\ga , r) , \quad p\le r \le \infty,
\\
&\dot   \cB^s_p(\ga , r)
 \hookrightarrow   \dot  F^s_{p,r}, \quad 1\le r\le p.
\end{align*}
\end{theorem}
This will be proved in \S\ref{sec:embeddings}. The statements  can be extended by combining them with the three  trivial embeddings for $q_1\le q_2$, $r_1\le r_2$, namely 
$\dot \cB^s_p(\gamma, r_1) \hookrightarrow \dot \cB^s_p(\gamma, r_2)$,
$\dot B^s_{q_1}[L^{p,r_1}]\hookrightarrow \dot B^s_{q_2}[L^{p,r_2}] $ and
$\dot F^s_{q_1}[L^{p,r_1}]\hookrightarrow \dot F^s_{q_2}[L^{p,r_2}] $.
Part (ii)  of the theorem  is an improvement and generalization over  Theorem 1.3 in \cite{gu-yung} which (in conjunction with our Theorem \ref{differences-thm-intro}) 
 yields that $\dot F^s_{p,2}\hookrightarrow \dot \cB^s_p(d, \infty)$ for $0<s<1$. Part (ii) also 
 covers  
 the  embedding $\dot C^s_p \hookrightarrow BSY^s_p\equiv \dot \cB^s_p(d,\infty)$ for the homogeneous Calder\'on-Campanato (or DeVore-Sharpley) spaces in \cite{devore-sharpley}, \cite{christ} which was obtained in \cite[Theorem 4.1]{dominguez-milman} for  $0<s<1$; indeed from \cite{se-TL} we know that  $\dot C^s_p=\dot F^s_{p,\infty}$ for  $0<s<1.$   For every $p\in (1,\infty)$  Theorem  \ref{thm:embeddings} also  recovers
 the known  embeddings 
 $\dot F^s_{p,r} \hookrightarrow \dot B^s_r[L^{p,r}]$ if  
   $p\le r$, 
  and 
  $ \dot B^s_r[L^{p,r}] \hookrightarrow \dot F^s_{p,r}$ if  
  $r\le p$;   \cf \cite[Theorem 1.2(iv), Theorem 1.1(iv)]{SeTr}. 
  
 
  In view of the case $r=\infty$ of the embedding in part (ii)
  of Theorem \ref{thm:embeddings} it is natural to ask whether in the embedding
  $\dot F^s_{p,\infty} \hookrightarrow \dot B^s_p(\gamma, \infty) $
  the  Triebel-Lizorkin space $\dot F^s_{p,\infty}$ can be replaced  by the larger Besov space $\dot B^s_{p,\infty}$; this was implicitly  suggested in \cite{poliakovsky}. Part (i) of the following theorem  implies  a negative answer, and in fact a stronger result.

\begin{theorem} \label{thm:non-emb} Let $s\in \bbR$, $1 < p < r \leq \infty$.
 Then the following hold.

(i) For all $\gamma\in \bbR$,
\[  \dot B^s_{p,r} \setminus  \dot \cB^s_p(\ga , \infty) \neq \emptyset.\]

(ii) For all $\beta, \gamma\in \bbR$ with $\beta \ne \gamma$,
\[\dot \cB^s_p(\beta,r)\setminus \dot \cB^s_p (\gamma,\infty) \neq\emptyset.\]

    \end{theorem}

This will be proved in \S\ref{sec:non-embeddings}, along with corresponding versions for the inhomogeneous spaces.

 Since $\dot \cB^s_p(\gamma,p)=\dot B^s_{p,p}$ for all $\ga \in \R$ (see \eqref{independence-of-gamma}) it is clear that the assumption $r>p$ is necessary in Theorem \ref{thm:non-emb}.
We also address the case  $\gamma=0$ in  part (ii) of Theorem \ref{thm:embeddings}; the following result shows that the condition $\gamma\neq 0$ is necessary for those statements.


\begin{theorem} \label{thm:non-emb-gamma} Let $s\in \bbR$ and $1<p<\infty$.
For the case $\gamma=0$ the following hold.

(i)   For all $r>p$ 
\[ \dot F^s_{p,r} \setminus \dot \cB^s_p(0,\infty)\neq \emptyset.\]

(ii) For all $r<p$ 
\[ \dot \cB^s_p(0,1) \setminus  \dot F^s_{p,r} \neq \emptyset. \]
\end{theorem}

\begin{remark}
By  part (ii) of Theorem \ref{thm:non-emb} we know 
that for $0<s<M$ and $\gamma_1\neq \gamma_2$ the seminorms 
$\|\cQ_{M,s+\gamma_i/p} f\|_{L^{p,\infty}(\nu_{\ga_i})}$, $i=1,2$ are not equivalent on the space of Schwartz functions.
This is in striking contrast with the  limiting result for 
$\mathscr D_{1+\gamma/p}$,
by Brezis and three of the authors  \cite{bsvy},  where it is shown that for $1<p<\infty$, and all $\gamma\neq 0$ the semi-norms $\|\mathscr D_{1+\gamma/p} f\|_{L^{p,\infty}(\nu_\ga)}$ are equivalent with the Gagliardo 
semi-norm $\|\nabla f\|_p$. Moreover, for  $p=1$ one has 
$\|f\|_{\dot{BV}}\approx \|\mathscr D_{ 1+\gamma} f\|_{L^{1,\infty}(\nu_\ga)} $ provided that $\gamma\in \bbR\setminus[-1,0]$ (and this additional assumption is necessary).  These equivalences hold under the a-priori assumption that
$f$ is locally integrable.  
\end{remark} 
\subsubsection*{An embedding result involving $\dot {BV}$}

Denote by  $V^\infty=V^\infty(\bbR^d)$ the quotient space of $L^{\infty}$ by additive constants, with norm 
\[\|f\|_{V^\infty}= \inf_{c\in \bbC}\|f-c\|_\infty.\] 
 Denote by $[\cdot,\cdot]_{\theta,r}$ the real interpolation spaces 
for the Peetre $K_{\theta,r}$ method \cite[Section 3.1]{BL}.
The following embedding result involves 
a real interpolation space between $\dot{BV}$ and $V^\infty$. It will be used below  to study solutions of harmonic and caloric functions on $\bbR^{d+1}_+$.

\begin{theorem} \label{thm:BVemb} Let $\gamma \in \bbR\setminus [-1,0]$ and $1<p<\infty$.  Then 
\[ [ V^\infty, \dot {BV}]_{\frac 1p,1} \hookrightarrow \dot \cB^{1/p}_p(\gamma,\infty).
\]
\end{theorem}
The case $\ga = d$ of Theorem \ref{thm:BVemb} has its roots in \cite[Theorem 1.4]{br-sch-yung-lor}. Its full generality is based on an estimate in \cite{bsvy}. It    complements  interpolation results in   \cite[Theorem 1.4]{cohen-et-al}, and extends an embedding theorem by Greco and Schiattarella \cite{greco-schiattarella} for functions of bounded variation on the unit circle. 

\subsection*{\it Harmonic and caloric functions in the upper half space} We now formulate some consequences of the embedding in Theorem \ref{thm:BVemb}. 
The original motivation of the space \(\dot\cB^{1/2}_2(1,\infty)\),  defined in terms of difference operators, came from the study of harmonic extension of functions of bounded variation in  \cite{greco-schiattarella}
(see also an earlier result by Iwaniec-Martin-Sbordone \cite{iwaniec-martin-sbordone}  for circle homeomorphisms). 
For a function 
such that 
\(\int_{\R^{d}} \abs{f (x)}(1 + \abs{x})^{-d-1} \dif x< \infty\),
the harmonic extension to the upper half space  $\bbR^{d+1}_+$ through the Poisson kernel is given 
by 
\[
 \poi f (x, t) = \frac{\Gamma(\frac{d + 1}{2})}{\pi^{\frac{d + 1}{2}}}
 \int_{\R^d} \frac{t}{(\abs{x - y}^2 + t^2)^{\frac{d + 1}{2}}} f (y) \dif y.
\]

 In order to state our result let  \Be\label{eq:Kb}\napoi^b f(x,t) = t^{1-b} \nabla\poi f(x,t) \Ee  where  $\nabla \poi$ denotes the $(x,t)$-gradient, for $t>0$,
 i.e.
 \[
\begin{split}
 \nabla \poi f (x, t)
 =
 &
 \frac{\Gamma(\frac{d + 1}{2})}{\pi^{\frac{d + 1}{2}} }
 \int_{\R^d} \frac{((d+1)t (x -y), \abs{x - y}^2 - d t^2)}{(\abs{x - y}^2 + t^2)^{\frac{d + 3}{2}}} (f (y) - f(x)) \dif y.
\end{split}
\]
This last expression makes sense for $f \in V^{\infty}+ \dot{BV}$. Define the measure  \(\lambda_\gamma\) on Lebesgue measurable sets of \(\R^{d+1}_+\) by \Be\label{eq:lagameas}
  \lambda_\gamma(E) = \iint_E \dif x \frac{\dif t}{t^{1-\gamma}}.
\Ee


\begin{corollary} \label{cor:Poisson-emb} Let $1<p<\infty$, $\gamma\in \bbR\setminus [-1,0] $. Then
\[\napoi^{\frac{\ga+1}{p}} : \, [ V^\infty, \dot {BV}]_{\frac 1p,1}\to L^{p,\infty}(\la_\ga)\] 
is bounded. In particular
\[\nabla\mathscr P  : \, [ V^\infty, \dot {BV}]_{\frac 12,1}\to L^{2,\infty}(\dif x\dif t)\] is bounded.
\end{corollary}
\begin{remark}When \(d = 1\) we have  \(\dot{BV} (\R) \hookrightarrow V^\infty (\R)\) 
and thus
we recover the upper half plane analogue of Theorem 4.2 of \cite{greco-schiattarella}, saying that  $\nabla \mathscr Pf\in L^{2,\infty}(\bbR^2_+)$ for $f\in \dot{BV}(\bbR)$.
\end{remark}

Another corollary is about  
  solutions $u(x,t)=Uf(x,t) =e^{t\Delta}f(x)$ of  the initial value problem for the heat equation in the upper half space,  \Be\label{eq:heat} \frac{\partial u}{\partial t} =\Delta u , \quad  u|_{t=0}=f. \Ee
For $b\in \bbR$, $t>0$, define  $\mathscr H^b =(\mathscr H_1^b,\dots,\mathscr H_{d+1}^b)$ by 
\begin{align*}
    \mathscr H_j^b f(x,t) &= t^{\frac{1}{2}-b} \frac{\partial}{\partial x_j }U\!f(x,t), \quad j=1,\dots, d
\\
\mathscr H_{d+1}^b f(x,t) &= t^{1-b} \frac{\partial}{\partial t}  U\!f (x,t) .
\end{align*}
\begin{corollary} \label{cor:heat}
Let $\beta\in \bbR\setminus [-\frac 12,0]$, and $1<p<\infty$.  
Then 

(i)
\(\mathscr H^{\frac{2\beta+1}{2p} } : [V^\infty,\dot{BV}]_{\frac 1p,1} \to L^{p,\infty} (\la_{\beta}) \) is bounded.

(ii) 
Let $u=Uf$ solve the problem \eqref{eq:heat} for $t>0$.
Then  
\begin{align*}
    f\in  [V^\infty, \dot{BV} ]_{\frac 23,1}  &\implies 
    \frac{\partial u}{\partial t} =\Delta_xu \in L^{\frac32,\infty} (\bbR^{d+1}_+, \dif x\dif t) ,
    \\
    f\in  [V^\infty, \dot{BV} ]_{\frac 13,1}  &\implies 
    \nabla_x u  \in L^{3,\infty} (\bbR^{d+1}_+, \dif x\dif t) .
\end{align*}
\end{corollary} 
When $d=1$  we 
obtain a  caloric analogue of the result in \cite{greco-schiattarella}, for boundary values in $\dot {BV}(\bbR)$.
\begin{corollary} Let $f\in \dot{BV}(\bbR)$ and let $u$ solve the
initial value problem $\frac{\partial u}{\partial t}=\frac{\partial^2 u}{\partial x^2}$, $u(x,0)=f(x)$. Then
$\frac{\partial u}{\partial t} \in L^{\frac32,\infty}(\bbR^2_+)$ and
$\frac{\partial u}{\partial x} \in L^{3,\infty}(\bbR^2_+)$.
\end{corollary}
 
\begin{remark} \label{rem:Vpinterpol} It would be interesting to upgrade the results of Theorem \ref{thm:BVemb} and/or the corollaries to  other interpolation spaces of $V^\infty$ and $\dot{BV}$. A related question in dimension $d=1$ is whether such  inequalities  can be proved for functions in the Wiener spaces $V^p$ of bounded $p$-variation. 
 Note that $V^1=\dot{BV} $ and that for $1<p<\infty$ we have  $[V^\infty,V^1]_{\frac 1p,p} \subset V^p $, see \cite{bergh-peetre}.
 If $V^{p,\infty}$ denotes the space of $f$ for which the numbers $\fN(f,\alpha)$ of $\alpha$-jumps satisfy $\sup_{\alpha>0}\alpha \fN(f,\alpha)^{1/p} <\infty$ then, by \cite{mirek-stein-zorin},  $V^p\subset V^{p,\infty} =  [V^\infty, V^1]_{1/p,\infty}$.
See also  \cite{cobos-kruglyak} for a related result on the $K$-functional for the couple $(V^\infty, V^1) $.
\end{remark}

\subsection*{Interpolation} We review the problem of interpolation of Besov spaces.
Recall 
 the definition of the  homogeneous Besov space $\dot B^s_{p,q}\equiv \dot B^{s}_{q} (L^p)$ as the subspace of 
 $f\in \cS_\infty'(\bbR^d) $ 
for which  \(\|f\|_{\dot B^s_{p,q}} \coloneqq(\sum_{k\in \bbZ} \|2^{ks}L_k f\|_p^q)^{1/q} \) is finite. 
Regarding real interpolation, the case for fixed $p$ and varying $s$ is well known.
Suppose $s_0, s_1 \in \R$ with $s_0\neq s_1$. If $1 \leq p, r \leq \infty$, one   has \cite[Theorem 6.4.5(i)]{BL}
\( [\dot B^{s_0}_{p,p},\dot  B^{s_1}_{p,p} ]_{\theta,r} = \dot B^{s}_{p,r}\) if 
$s=(1-\theta)s_0+\theta s_1,$ $\theta \in (0,1)$; see also \cite{devore-popov}.
For the case  $p_0\neq p_1$ the spaces $\dot \cB^{s}_{p}(\ga,r)$ arise as  interpolation spaces for the $K_{\theta,r}$-method.
The following theorem and corollary were already known to Krepkogorski\u\i{} \cite{krepkogorskii} who considered the inhomogeneous variants. For an extension to the quasi-Banach range see 
\cite{krepkogorskii2}. For a description of the interpolation spaces via wavelet coefficients see  also  the recent work by Besoy, Haroske, and  Triebel    \cite{besoy-haroske-triebel}. 

\begin{theorem}\label{interpol-thm-Fourier} 
Let $1\le p_0,p_1, r\le \infty $, $p_0\neq p_1$, $s_0,s_1\in \bbR$.  
Let \begin{equation}
\label{eq_gamma}
\ga=- \frac{s_0-s_1} 
{\tfrac 1{p_0}-\tfrac 1{p_1}} \,,
\end{equation}
  let $0<\theta<1$ and 
\begin{equation}\label{eq_theta} (\tfrac  1p,s)=(1-\theta) (\tfrac 1{p_0},s_0)+\theta(\tfrac 1{p_1},s_1).
\end{equation}
Then
 \begin{align} \label{eq:interpol_nice}
  [ \dot B^{s_0}_{p_0,p_0}, \dot B^{s_1}_{p_1,p_1}]_{\theta,r} &=
 \dot \cB^s_p(\ga,r), 
 \end{align}
 with equivalence of (quasi-)norms. \qed
 \end{theorem}
\medskip

\begin{corollary}\label{cor:TLinterpol}
Let $1<p_0,p_1< \infty $, $p_0\neq p_1$, $s_0,s_1\in \bbR$, $1\le q_0, q_1,r_0, r_1\le \infty$ and $1\le r\le\infty$. Suppose that \eqref{eq_gamma} and \eqref{eq_theta} hold with $0<\theta<1$.  Then  
\begin{subequations}
\Be \label{eq:cBinterpol}
  [ \dot\cB^{s_0}_{p_0} (\gamma,r_0), \dot \cB^{s_1}_{p_1} (\gamma, r_1)]_{\theta,r} =
 \dot \cB^s_p(\ga,r).
\Ee
Moreover, if $s_0\neq s_1$,
\Be \label{eq:TLinterpol}
    [ \dot F^{s_0}_{p_0,q_0}, \dot F^{s_1}_{p_1,q_1}]_{\theta,r} =
 \dot \cB^s_p(\ga,r). \Ee
  \end{subequations}
\end{corollary}
Note that for $s\in \bbN\cup \{0\}$ and $1<p<\infty$, the space $\dot F^s_{p,2}$ is identified with the Sobolev space $\dot{W}^{s,p}$. Thus, if $s_0,s_1$ are non-negative integers, 
 $s_0 \ne s_1$, and $1 < p_0, p_1 < \infty$ with $p_0 \ne p_1$, then for $0 < \theta < 1$ and $1 \leq r \leq \infty$ we get in particular $$[\dot{W}^{s_0,p_0},\dot{W}^{s_1,p_1}]_{\theta,r} = \dot \cB^s_p(\ga,r)$$ where $(\tfrac{1}{p},s)$ and $\ga$ are given by  \eqref{eq_theta} and \eqref{eq_gamma}. 

For completeness we shall sketch in  \S\ref{sec:interpolation}
the standard proofs based on the Fourier analytic definition which are very much  analogous to \cite{krepkogorskii}. More interestingly, for  $M=1$ and   $s_0,s_1\in (0,1)$, an alternative approach to the interpolation result \eqref{eq:interpol_nice} will be given in \S\ref{sec:interpol-diff}, based directly on the characterization via first order differences. 


\subsection*{Nonlinear wavelet approximation} 
Our results  can be obtained to prove new results on best approximation via $n$ terms in a wavelet basis, relating it to suitable regularity properties of the given function.

To fix ideas we first recall basic notation in wavelet theory.
Let $u\in \bbN$, $\phi\in C^u(\bbR)$ be a univariate scaling function associated with the  univariate wavelet $\psi\in C^u(\bbR)$.
Let $\psi^0 := \phi$ and $\psi^1 := \psi$. If $E$ denote the set of the $2^d-1$ non-zero vertices of $[0,1]^d$, given $e = (e_1, \ldots, e_d) \in E$, we define the $d$-variate wavelets $\psi^{e}(x)=\prod_{i=1}^d \psi^{e_i} (x_i)$.
As in \cite{kyriazis} 
we assume certain decay and nonvanishing moment conditions on the $\psi^{e}$, namely 
\begin{subequations}
\begin{equation}\label{Cond1}
	\sup_{x \in \R^d} (1 + |x|)^M |D^\alpha \psi^{e} (x)| < \infty, \qquad |\alpha| \leq u,  \qquad e \in E,
\end{equation}
and 
\begin{equation}\label{Cond2}
	\int_{\R^d} x^\alpha \psi^{e}(x)  \dif x = 0, \qquad |\alpha| < u, \qquad e \in E,
\end{equation} for $u$, $M$ satisfying 
\Be\label{eq:u-M-choice}
u>|s|, \quad M>d+u.
\Ee If one works with $L^p$-based Besov spaces and allows the parameter range to be $p>0$, then one needs to require
	$u > \max \{\frac{d}{\min\{1,p\}}-d-s, s\}$, $M > \max\{ \frac{d}{\min\{1, p\}}, d + u\}$.
	\end{subequations}

Let, for $j\in \bbZ$ and $m\in \bbZ^d$, \begin{equation}\label{Tensor}
	\psi^{e}_{j, m}(x) := 2^{\frac{j d}{2}} \psi^{e}(2^j x - m).
\end{equation}
We assume that  the system 
\begin{equation}\label{WaveletSystem}
	\Psi = \{\psi^{e}_{j, m} : \, j \in \Z, \, m \in \Z^d, \, e \in E\}
\end{equation}
forms an {\it orthonormal basis in $L^2(\R^d)$}, see e.g.\ \cite{daubechies} for an introduction to wavelet theory.

Let $1<q<\infty$. Consider now the \emph{best $n$-term approximation of $f\in L^q(\bbR^d)$,  with respect to $\Psi$}, measured in the $L^q(\bbR^d)$ norm;  i.e.,
$$
	\sigma_n(f)_q = \inf \Big\{\Big\|f- \sum_{\psi_\nu \in \Lambda\subset \Psi} c_\nu \psi_\nu \Big\|_{L^q(\R^d)} :  \#(\Lambda)  \leq n, \, c_\nu  \in \mathbb{C} \Big\}.  
$$ 
Let $\alpha > 0$
and $0 < r \leq \infty$. The related \emph{approximation space} $\mathcal{A}^\alpha_r(L^q, \Psi)$ is defined as the set of functions  $f\in L^q(\bbR^d)$   for which 
$$
	\|f\|_{\mathcal{A}^\alpha_r(L^q, \Psi)} =\begin{cases}  \Big(\sum_{n=1}^\infty\big [n^{\alpha} \sigma_n(f)_q\big]^r \frac{1}{n} \Big)^{\frac{1}{r}}  &\text{if $r<\infty$}
	\\
	\sup_n n^\alpha \sigma_n(f)_q    &\text{if $r=\infty$}
	\end{cases}
$$
is finite.

It is well known that $\mathcal{A}^\alpha_r(L^q(\R^d), \Psi)$ can be characterized in terms of a certain interpolation space between $L^q(\R^d)$ and Besov spaces. Specifically, let  $1<q<\infty$, $0 < r \leq \infty$,  and $0 < s < \sigma.$ Then 
\Be\label{eq:DeVore-interpol}
	 \mathcal{A}^{s/d}_r (L^q, \Psi) = [L^q, \dot{B}^{\sigma}_{u, u}]_{\theta, r} \, \quad \text{ if  } \theta=\frac s\sigma \text{ and } \,\,\frac{1}{u} = \frac{1}{q}+ \frac{\sigma}{d} ;
\Ee
see DeVore's survey
\cite[(7.41)]{DeVore} and also  \cite[page 223]{Petrushev} for related results on spline approximation with $d=1$. 
We specialize  \eqref{eq:TLinterpol}    with $s_0 = 0$, $p_0= q$, $q_0= 2$,  $q_1 = p_1 = u$, $s_1=\sigma$, hence  $ \gamma= -\frac{s_1-s_0}{p_1^{-1}-p_0^{-1}}=-d $.
We thus see  that  for $\theta$, $u$ as in \eqref{eq:DeVore-interpol}  the space 
$[L^q, \dot{B}^{\sigma}_{u, u}]_{\theta, r}$ coincides with 
$\dot \cB_p^s(-d,r)$ if $\frac 1p=\frac 1q+\frac sd$.
Combining this with
\eqref{eq:DeVore-interpol}, 
we have verified
\begin{theorem} \label{thm:appr}
Let $1<q<\infty$, $0<s<d(1-\frac 1q)$, and  let 
$\frac 1p=\frac 1q+\frac sd$.
Then,  for $1\le r\le \infty$,
\[ \cA^{s/d}_r(L^q,\Psi) = \dot\cB^s_p(-d, r).\]
\end{theorem} 
For $r=p$, $0<s<d (1-\frac 1q)$    we recover 
$\cA^{s/d}_p(L^q,\Psi) = \dot B^{s}_{p,p}$ for $\frac 1p=\frac 1q+\frac sd$, which is a result proved by 
DeVore, Jawerth and Popov \cite{DJP}.
Together with our characterization in Theorem  \ref{differences-thm-intro} we achieve a new  interpretation via difference operators   of some results 
in \cite{DeVore, DeVorePetrovaTemlyakov, Gribonval, Kerkyacharian} where the spaces  $\mathcal{A}^\alpha_r(L^q, \Psi)$ are  characterized in terms of wavelet coefficients.

 For $r=\infty$ the spaces  $\mathcal{A}^{s/d}_\infty(L^q, \Psi)$ are  of special interest in applications, see 
 for example \cite{HS,CDD-elliptic, hansen}.
 In the statistics literature these spaces are sometimes referred to as `weak-Besov spaces' (see \cite{ACF, Rivoirard} and the references within). In view of Theorem  \ref{thm:appr}, these weak-Besov spaces coincide with $\mathcal{\dot{B}}^s_p(-d, \infty)$, with $s=d(1/p-1/q)$, i.e. $p=\frac{dq}{d+sq}$.
Putting $\alpha=s/d$ and combining  
Theorem \ref{differences-thm-intro} and Theorem \ref{thm:appr} we obtain 
\begin{corollary} \label{cor:appr-diff} Let $1<q<\infty$, $0<\alpha<1-\frac 1q$.
Then, for $M>\alpha d$,
\[\sup_{n\ge 1} n^{\alpha} \sigma_n(f)_q 
\approx \sup_{\la>0}\la\Big( \int \meas\Big(\Big\{x: |h|^{\frac dq}  |\Delta_h^M f(x)| >\lambda \Big\}\Big) \frac{\dif h}{|h|^{2d}}\Big) ^{\frac 1q+\alpha} \,.\]
\end{corollary}

\begin{remarka}\label{rem:appr-diff}  There are suitable extensions of the definitions of this paper,  and many  of the results,   to certain   parameter ranges  in   the quasi-Banach setting (that is, to the cases $r<1$ and $p\le 1$);  we intend to pursue these elsewhere. In particular it is interesting to extend Theorem \ref{thm:appr} to values of  $s\ge d(1-1/q)$ and  $r>0$;  this requires  consideration of  the spaces $\dot \cB^s_p(-d,r)$  in the range $p\le 1$. 
\end{remarka}

\subsubsection*{Notation}  We denote by $\mathscr L^d(E)$ the Lebesgue measure of a Lebesgue measurable set in $\bbR^d$, and also write 
$\meas\,E$ for $\mathscr L^d(E)$ when the dimension is clear from the context. A measurable function $f:\bbR^d\to \bbC$ will always be assumed to be defined almost everywhere. We use $\widehat f(\xi)= \int f(y) e^{-i\inn{y}{\xi}} \dif y$ as definition of the Fourier transform. For a function $m$ on $\widehat \bbR^d$ we define $m(D)$ to be the convolution operator with Fourier multiplier $m$, i.e. it is given by $\widehat{m(D)f}(\xi)=m(\xi)\widehat f(\xi)$.
We let $C^\infty_c$ be the space of compactly supported $C^\infty$-functions, $\cS$ be the space of Schwartz functions, and $\cS_M$ be the subspace of $\cS$ consisting of those Schwartz functions whose moments up to order $M-1$ vanish. Also let $\cS_\infty=\bigcap_{M\in \bbN} \cS_M$. We denote by $\cS'$ the space of tempered distributions and by $\cS_M'$, $\cS_\infty'$  the dual spaces of $\cS_M$ and $\cS_\infty$, respectively.
We let, for $k\in \bbZ$,  $L_k=\varphi(2^{-k}D)$  and $\widetilde L_k=\widetilde \varphi(2^{-k}D)$ be operators in frequency localizing  Littlewood-Paley decompositions, satisfying $\widetilde L_k L_k=L_k$. The functions $\varphi$, $\widetilde \varphi$ are radial and the relevant properties are defined in  \eqref{eq:varphiassumptions} and \eqref{eq:Tildephi-assu}, respectively. For a set $E$ with positive measure, the slashed integral $\fint_E f$ is used to denote the average of $f$ over $E$.

\subsubsection*{Structure of the paper} In \S\ref{sec:norm-equivalence} we prove a rudimentary form of the characterization in Theorem \ref{differences-thm-intro} just for $\cS_\infty$ functions. The full proof of Theorem \ref{differences-thm-intro} will be given in \S\ref{sec:Proof-of-equiv}. 
The embedding results in Theorem \ref{thm:embeddings}
are proved in \S\ref{sec:embeddings}. Various counterexamples establishing Theorem \ref{thm:non-emb} are discussed in \S\ref{sec:non-embeddings}. In \S\ref{sec:BV-interpol} 
we give the proof of Theorem \ref{thm:BVemb}  and in
\S\ref{sec:harmonic-caloric}
the proof of Corollaries \ref{cor:Poisson-emb} and \ref{cor:heat}.
In \S\ref{sec:interpolation} we include a  proof of Theorem \ref{interpol-thm-Fourier} based only on the Fourier analytic definition of $\dot\cB_p^s(\ga,r)$. A different proof of the interpolation result, just for parameters $s_i\in (0,1)$ and based on a retraction argument using first order differences  is given in \S\ref{sec:interpol-diff}. 

\subsubsection*{Acknowledgements} 
A.S., J.V.S. and P.-L.Y.  would like to thank Ha\"im Brezis for  collaboration on related topics.  A.S. is also grateful to him for a kind exchange in which he raised one of the   questions addressed in this paper (prompted by 
his online  distinguished lecture 
at the University of Connecticut in April 2021). 
 A.S. and P.-L.Y.    thank
the  Hausdorff Research Institute of Mathematics and the organizers of the trimester program ``Harmonic Analysis and Analytic Number Theory'' for a pleasant working environment in  the summer of  2021. 
The  research was  supported in part by  the French National Research Agency (ANR-10-LABX-0070), (ANR-11-IDEX-0007)    and by MTM2017-84058-P (AEI/FEDER, UE)   (\'O.D.), 
by 
  NSF grants  DMS-2054220 (A.S.) and   DMS-1764265 (B.S.), and by a Future Fellowship FT200100399 from the Australian Research Council (P.-L.Y.).  We thank an unknown referee for a thoughtful report.

\section{ Norm equivalences \texorpdfstring{for  $\cS_\infty$-functions}{}}
\label{sec:norm-equivalence}
Before giving the full proof of Theorem  \ref{differences-thm-intro} we give a proof of the norm equivalence  for functions in the class $\cS_\infty (\bbR^d)$.
Note that for $f\in \cS_\infty(\bbR^d)$ we have $f=\sum_{k\in \bbZ} L_k f$ with convergence in the  topology of $\cS(\bbR^d)$.
\begin{proposition}
\label{prop:equiv-Schw}
Let $M \in \N$, $1<p<\infty$, $1\le r\le \infty$, $\ga \in \R$ and $0<s<M$. 
For $f\in \cS_\infty(\bbR^d)$, 
\Be \label{eq:equiv-onS0} \|f\|_{\dot\cB^s_p(\ga,r)} \approx \|\cQ_{M, s+\ga/p} f\|_{L^{p,r}(\nu_\gamma)}.
\Ee
\end{proposition}

Let $\Tilde \varphi \in C^{\infty}_c(\R^d)$ be such that
\Be\label{eq:Tildephi-assu} 
\supp(\Tilde \varphi) \subset \{\xi: 1/2<|\xi|< 2\}  \text {  and } \Tilde \varphi(\xi)=1 \text{ for $3/4\le |\xi|\le 7/4$}. 
\Ee
This implies $\Tilde \varphi \varphi=\varphi$.
Let $\widetilde L_k=\Tilde \varphi(2^{-k}D) $ so that 
    $L_k=\widetilde L_k L_k =L_k\widetilde L_k$.
    
    To bound  \(\|\cQ_{M, s+\ga/p} f\|_{L^{p,r}(\nu_\gamma) } \) in terms of \(\|f\|_{\dot \cB^s_p(\gamma,r)}\), 
we use the following lemma. 
\begin{lemma} \label{lem:2.2}
Let $M \in \N$, $1 < p < \infty$, $1 \leq r \leq \infty$, $b, \ga \in \R$ with $0 < b-\frac{\gamma}{p} < M$. Then the operator
\[
T_b g(x,h) \coloneqq \sum_{k \in \Z} \frac{\Delta_h^M \tilde{L}_k g(x,k)}{(2^k|h|)^b}, \quad (x,h) \in \R^d \times (\R^d \setminus \{0\})
\]
defines a bounded linear map from $L^{p,r}(\mu_{\gamma})$ to $L^{p,r}(\nu_{\ga})$. 
\end{lemma}

\begin{proof}
By real interpolation it suffices to consider the case $ r=p$. From the elementary inequality
\[
\|\Delta_h^M \tilde{L}_k\|_{L^p \to L^p} \lesssim \min\{1,(2^k |h|)^M\},
\]
we obtain
\begin{align*}
&\|T_b g\|_{L^p(\nu_{\ga})}
\lesssim 
\Big (\int \Big[ \sum_{k \in \Z} \frac{\|\Delta_h^M \tilde{L}_k g(\cdot,k) \|_{L^p(\dif x)}}{(2^k |h|)^b} \Big ]^p \frac{\dif h}{|h|^{d-\ga}} \Big)^{1/p} \\
&\lesssim \Big (\int\Big[  \sum_{k \in \Z} \min\{(2^k|h|)^{-(b-\frac{\ga}{p})},(2^k |h|)^{M-(b-\frac{\ga}{p})}\} 2^{-k\frac{\ga}{p}} \|g(\cdot,k)\|_{p}\Big ]^p\frac{\dif h}{|h|^d} \Big)^{1/p}  \\
&\simeq \Big(\sum_{j\in \bbZ} \Big[ \sum_{k \in \Z} \min\{(2^{k-j})^{-(b-\frac{\ga}{p})},(2^{k-j})^{M-(b-\frac{\ga}{p})}\} 2^{-k\frac{\ga}{p}} \|g(\cdot,k)\|_p\Big]^p \Big)^{1/p}  
\end{align*}
and the desired conclusion follows since if $\alpha, \beta > 0$ then the convolution on $\Z$ with the sequence $\{\min\{2^{-k \alpha},2^{k\beta}\}\}_{k \in \Z} \in \ell^1(\Z)$ is bounded on $\ell^p(\Z)$. 
\end{proof}

To apply the lemma note that under the hypotheses of Proposition~\ref{prop:equiv-Schw} we have
    \Be \label{eq:representation}
    f(x) = \sum_{k \in \Z} 2^{-k(s+\frac{\gamma}{p})} \tilde{L}_k P^{s+ \frac{\gamma}{p}} f(x,k)
    \Ee
    with the convergence in $\cS_{\infty}(\R^d)$ (in particular pointwise for every $x \in \R^d$). Hence for every $(x,h) \in \R^d \times (\R^d \setminus \{0\})$ we have
\[
\frac{\Delta_h^M f(x)}{|h|^{s+\frac{\gamma}{p}}} = T_{s+\frac{\gamma}{p}} P^{s+\frac{\gamma}{p}} f(x,h).
\]
Lemma \ref{lem:2.2} with $b \coloneqq s + \frac{\gamma}{p}$ (which satisfies $0 < b-\frac{\gamma}{p} < M$) and $g \coloneqq P^{s+\frac{\gamma}{p}}f$ yields the  inequality
    \Be
    \label{eq:goal2.2}\Big\|\Big\{\frac{\Delta_h^M f(x)}{|h|^{s+\frac{\gamma}{p}}} \Big\}\Big\|_{L^{p,r}(\nu_{\gamma})} \lesssim \|P^{s+\frac{\gamma}{p}} f\|_{L^{p,r}(\mu_{\gamma})}
\Ee
and thus the following corollary.

\begin{corollary} \label{cor:direct-ineq}
Let $M \in \N$, $1<p<\infty$, $1\le r\le\infty$, $\ga\in \bbR$ and $0<s<M$.  Then for $f\in \cS_\infty(\bbR^d)$ 
    \[\Big\|\Big\{\frac{\Delta_h^M f(x)}{|h|^{s+\frac{\gamma}{p}}} \Big\}\Big\|_{L^{p,r}(\nu_{\gamma})} \lesssim
    \|f\|_{\dot\cB^s_p(\ga,r) } .\]
\end{corollary}

For the converse inequality we like to consider an operator acting on $F(x,h)=|h|^{-b}\Delta_h^M f$, for $b=s+\ga/p$, and then we are faced with the task of ``dividing out" the difference operator. To achieve this we work with the partition of unity of the annulus $\{\xi\in \bbR^d: 1/2<|\xi|<2\}$.
Alternative Fourier arguments can be found e.g. in 
\cite[5.2.1]{nikolskii}.

Let $\eps<(10 M)^{-1}$.
We use a finite partition $\{\chi_\ka\}_{\ka=1}^N$ of unity on the support   of $\varphi$, so that 
$\chi_\ka\in C^\infty_c$ is supported on the ball $B^d(u_\ka,\eps) $. Let $w_\ka  =
\frac{\pi u_\ka}{2 |u_\ka|^2}$ and then  we have, for 
$\xi\in \supp(\chi_\ka)$ and $ |w-  w_\ka|\le \eps $,
\begin{align*} |\inn{\xi}{ w}-\frac \pi 2|\le |\inn{\xi}{w-w_\ka} |+ |\inn{\xi-u_\ka}{w_\ka}|+ |\inn{u_\ka}{w_\ka}-\tfrac \pi 2|\le 2\eps+2\eps+0.
\end{align*} 
We may then write
\begin{subequations}\label{eq:varphi-expression}
\Be \label{eq:decomp-of-phi}\varphi(\xi)= \sum_{\ka=1}^N
m_\ka(\xi) {\int_{|h-w_\ka|\le \eps} (e^{i\inn{\xi}{h}} -1)^M \dif h} \Ee
where
\Be \label{eq:def-of-mkappa}m_\ka(\xi)= \varphi(\xi) 
\frac{\chi_\ka(\xi) }{\int_{|h-w_\ka|\le \eps} (e^{i\inn{\xi}{h}} -1)^M \dif h}.
\Ee 
\end{subequations}
Since the denominator is bounded away from $0$ on the support of $\chi_\ka$ we get  $|\partial^\alpha m_\ka(\xi)|\le C_\alpha$ for all multiindices $\alpha$, and thus 
the $L^1$ norms of the Fourier inverse transforms of the $m_\ka$ are finite. 
We then get 
\Be\label{eq-resolution-of-f}
 L_jf= \sum_{\ka=1}^N   m_\ka(2^{-j}D) 
\int_{| h-2^{-j}w_\ka|\le \eps 2^{-j}} \Delta_h^M \!f \, \frac{\dif h }{2^{-jd} } .
\Ee
\begin{lemma} \label{lem:V-lemma} 
Let $m$ be the Fourier transform of a bounded Borel measure, with $L^1\to L^1$ multiplier norm $\|m\|_{M_1}$. Let $w\in  \bbR^d$ such that $1/2\le |w|\le 2$ and $\eps\in (0,\tfrac 12)$.
For $b, \ga \in \R$, $1 < p < \infty$, $1 \leq r \leq \infty$, and $F\in L^{p,r}(\bbR^d\times(\bbR^d\setminus\{0\}), \nu_\gamma) $ define 
$ V^b_{m,w,\eps}F$  by
\begin{multline}\label{eq:Vbka-def} 
V^b_{m,w,\eps} F(\cdot ,k) \equiv V^{b,k}_{m,w,\eps}F\\= m(2^{-k}D) \int_{\substack {|h-2^{-k}w| \le \eps 2^{-k} }} (2^k|h|)^b F(\cdot ,h) \frac{\dif h}{2^{-kd}}.
\end{multline}
Then $V^b_{m,w,\eps} $  maps $L^{p,r}(\nu_\gamma)$ to $L^{p,r}(\mu_\gamma)$ and we have 
\Be \| V^b_{m,w,\eps} F\|_{L^{p,r}(\mu_\ga)} \le C \|m\|_{M^1} \|F\|_{L^{p,r} (\nu_\ga)}\Ee
where $C$ only depends on $p$, $r$, $b$, $\ga$.
\end{lemma}

\begin{proof}
Since $(F,m)\mapsto V^b_{m,w,\eps}F$ is bilinear we may normalize and assume that $\|m\|_{M_1}=1$.
Again by real interpolation it suffices to prove the theorem for $p=r$, $1\le p\le \infty$.
Since  $\|m(2^{-k} D)\|_{L^p\to L^p} \le 1$ for $1 \leq p \leq \infty$ we  obtain 
\begin{align*}
    \|V_{m,w,\eps}^b F\|_{L^p(\mu_\gamma)}&\lc 
    \Big(\sum_{k\in \bbZ}  2^{-k\ga} \Big\|
    \int_{\substack {|h-2^{-k}w| \le \eps 2^{-k} }} (2^k|h|)^b F(\cdot,h) \frac{\dif h}{2^{-kd}}\Big\|_p^p\Big)^{1/p}
    \\
    &\lc 
    \Big(\sum_{k\in \bbZ}  2^{-k\ga} \Big\|\int_{2^{-k-1}\le|h|\le 2^{-k+1} }| F(\cdot,h) |\frac{\dif h}{|h|^d}\Big\|_p^p\Big)^{1/p}
    \\
    &\lc 
    \Big(\int_{\R^d} \| F(\cdot,h) \|_p^p \frac{\dif h}{|h|^{d-\gamma} }\Big)^{1/p} 
= \|F\|_{L^p(\nu_\gamma)}
\end{align*}
which completes the proof of the lemma. 
\end{proof}

To apply the lemma it is  beneficial to express $P^bf(\cdot,k)=2^{kb}L_kf $  as \begin{subequations} \label{eq:Pb-formulas}
\Be\label{eq:Pb-decomp}  P^bf(x,k)=\sum_{\ka=1}^N 
m_\ka(2^{-k} D) \int_{| h-2^{-k}w_\ka|\le \eps 2^{-k}} (2^k |h|)^b\frac{\Delta_h^M f (x)}{|h|^b} \frac{\dif h }{2^{-kd} } 
\Ee
and thus we get 
\Be\label{eq:expandPb} P^b f =\sum_{\ka=1}^N V^b_{m_\ka, w_\ka,\eps} F,  \text{ with }  F(x,h) = \frac{\Delta_h^Mf(x)}{|h|^b}. \Ee
\end{subequations}
Now, setting $b=s+\ga/p$, Lemma \ref{lem:V-lemma} yields 
\begin{corollary} \label{cor:Besovupperbound}Let  $M \in \N$, $1<p<\infty$, $1\le r\le \infty$, $s , \ga \in \bbR$.
For  $f\in \cS_\infty (\bbR^d)$,
\[ \|f\|_{\dot \cB^s_p(\gamma,r)} \lc \|\cQ_{M, s+\frac \ga p} f\|_{L^{p,r}(\nu_\gamma) }. \] 
\end{corollary} 
Proposition \ref{prop:equiv-Schw} is just the combination of Corollaries \ref{cor:direct-ineq} and 
\ref{cor:Besovupperbound}. 

\section{Norm equivalences for all measurable functions}
\label{sec:Proof-of-equiv}
We give the proof of Theorem \ref{differences-thm-intro}. We begin by rephrasing it in a more abstract way which allows us to keep in mind the distinction between equivalence classes modulo all polynomials and modulo polynomials of degree $<M$.
Let $\cM$ denote the space of (Lebesgue almost everywhere equivalence classes of) measurable functions on $\bbR^d$
and let $\cP_{M}$  denote the space of (almost everywhere equivalence classes of) functions which are almost everywhere equal to a polynomial
of degree  at most $M$.
Let $\cM_M\coloneqq\cM/\cP_{M-1}$  and let $\pi_M:\cM\rightarrow \cM_M$ denote the projection map. Since the operators $\cQ_{M,s+\ga /p} $ annihilate polynomials of degree $\le M-1$ we can make the following definition.
\begin{defn}\label{def:fBspace}
    For $M\in \N$, $s\in \R$,
    $1<p<\infty$, and $1\le  r\leq \infty$, we define an extended
    norm\footnote{A priori, \(\| \cdot\|_{\fB_{M,s,p} (\ga,r)}\) is merely an extended semi-norm.  Lemma \ref{Lemma::Diff::AlmostPoly} below shows
    that \(\| \pi_M f\|_{\fB_{M,s,p} (\ga,r)}=0\Leftrightarrow \pi_M f=0\).} on $\cM_{M}$  by
    \begin{equation*}
   \| \pi_{M}f\|_{\fB_{M,s,p} (\ga,r)} \coloneqq \|\cQ_{M,s+\frac \ga p} f\|_{L^{p,r}(\nu_\ga)}
            \end{equation*}
    and let $\fB_{M,s,p}(\ga,r)$ be the subspace of  $\cM_{M}$ for which  $\| \pi_{M}f\|_{\fB_{M,s,p}(\ga,r)}$ is finite.
\end{defn}

Recall that $\cT\subseteq \cM$  denotes 
the space of (Lebesgue almost everywhere equivalence classes of) tempered functions on \(\R^d\), and let   $\cT_M\coloneqq \cT/\cP_{M-1} $. 
$\pi_M \colon \cM \to \cM_M$ restricts to a map $\pi_M \colon \cT \to \cT_M$.
We let $\iota_M$ denote the natural map
$\cT_M \rightarrow  \cS_\infty'(\bbR^d)$ which assigns to $\pi_M(f)$ (with $f \in \cT$) the linear functional $\iota_M(\pi_M(f)):\phi\mapsto \int_{\R^d} f(x)\phi(x) \dif x$.
We rephrase Theorem \ref{differences-thm-intro} in the following, equivalent, form:

\begin{theorem}\label{Thm::Diff::MainDiffThm}
    Fix \(M\in \N\), \(M\geq 1\).
    For  $0<s<M$, $p\in (1,\infty)$, $r\in [1,\infty]$, \(\gamma\in \R\), we have
\begin{enumerate}[(i)]
    \item\label{Item::Diff::IncTM} \(\fB_{M,s,p} (\ga,r) \subseteq \cT_M\);
    \item\label{Item::Diff::InIota} \(\dot\cB^s_p(\ga,r) = \iota_M\mleft( \fB_{M,s,p} (\ga,r) \mright)\); 
    \item\label{Item::Diff::Bijection} The map
        \begin{equation*}
            \iota_M\big|_{\fB_{M,s,p} (\ga,r)} : \fB_{M,s,p} (\ga,r) \rightarrow \dot\cB^s_p(\ga,r)
        \end{equation*}
        is an isomorphism of normed vector spaces;
        i.e., it is 
         a bounded, bijective linear map with bounded inverse. 
\end{enumerate}
\end{theorem}

The rest of this section is devoted to the proof of Theorem \ref{Thm::Diff::MainDiffThm}. In what follows 
we denote, for  $M\in \bbN$, by  $\cS_M(\bbR^d)$  the closed subspace of $\cS(\bbR^d)$ which consists of all $f\in \cS(\bbR^d)$ with $\int p(x)f(x) \dif x=0$ for all polynomials of degree $\le M-1$. Then clearly
$\cS_\infty=\bigcap_{M\in \bbN} \cS_M$ (and  $\cS_\infty\equiv\cZ$ in the notation of \cite{triebel}).
We denote by $\cS_M'$ the dual space of $\cS_M$.
To prove the theorem, we introduce two maps:  \(\dot\cB^s_p(\ga,r)\leftrightarrow \fB_{M,s,p} (\ga,r)\),
which will turn out to be inverses to each other.
We begin with the map \(\dot\cB^s_p(\ga,r)\rightarrow \fB_{M,s,p} (\ga,r)\). The following proposition
is similar  to results  of Bourdaud \cite{bourdaud-besov} and Moussai  \cite{madani} for the so-called realized  Besov spaces, and, in fact, could be deduced from their results by interpolation arguments.

\begin{proposition}\label{Prop::Diff::PropG}
    Fix \(M\in \N\), \(M\geq 1\).  For \(0<s<M\), \(1<p<\infty\), \(\gamma\in \R\), and \(1\leq r\leq \infty\), there is a bounded linear map
    \begin{equation*}
        \EMap:\dot\cB^s_p(\ga,r)\rightarrow \fB_{M,s,p} (\ga,r)
    \end{equation*}
    such that \(\EMap(\dot\cB^s_p(\ga,r))\subseteq \cT_M\) and \(\iota_M\) is a left inverse to \(\EMap\); i.e.,
    \(\iota_M\EMap\) is the identity map \(\dot\cB^s_p(\ga,r)\rightarrow \dot\cB^s_p(\ga,r)\).
\end{proposition}

 We need a lemma  about the Littlewood-Paley decomposition 
 for $f\in \dot \cB^s_p(\ga,\infty)\subseteq \cS_\infty'(\bbR^d)$, for $p\in (1,\infty). $
  Note that $L_k f$  is a convolution of an element of \( \cS_\infty'(\bbR^d)\) and an element of \(\cS_\infty(\bbR^d)\), and thus a $C^\infty$-function. 
    By the definition of \(\dot \cB^s_p(\ga,\infty)\), \(L_k f\in L^{p,\infty}(\bbR^d)\) with $\|2^{ks} L_k f\|_{L^{p,\infty}(\R^d)} \lesssim \|f\|_{\dot \cB^s_p(\ga,\infty)}$ uniformly in $k \in \Z$.

    By Young's convolution inequality 
    \[\|\widetilde L_k\|_{L^{p,\infty} \to L^\infty} =O(2^{kd/p}) \] and from $\widetilde L_kL_k=L_k$ 
    we obtain 
    $\|L_k f\|_\infty \lc 2^{kd/p} \|L_k f\|_{L^{p,\infty}}$.
    We use this to establish convergence of the Littlewood-Paley decomposition in $\cS_M$, under the additional condition  $M>s-d/p$.

\begin{lemma}\label{lem:LP-SMdistr}
Let $M$ be a nonnegative integer, $1<p<\infty$, $N\in \bbN$. Then the following holds.

(i) For $f\in \dot \cB_p^s(\ga,\infty)$ and $\psi\in \cS_M$,
\[|\inn{L_jf}{\psi}| \le C_{N,M,\psi} 2^{j(\frac dp-s)}\min\{2^{-jN}, 2^{jM}\} \|f\|_{\dot\cB^s_p(\ga,\infty)}.
\]

(ii) Let 
$M>s-d/p$, and $f\in \dot\cB_p^s(\ga,r)$. Then $\sum _{j\in\bbZ}L_j f$ converges in $\cS_M'$.
\end{lemma}
\begin{proof}  
Since $\psi \in \cS$ we get $\|\widetilde L_j\psi\|_1\lc C_{N,\psi} 2^{-jN} $ for $j\ge 0$. Using the $M-1$ vanishing moment conditions we get $\|\widetilde L_j \psi\|_1 \lc 2^{jM}$ for $j\le 0$.

We have 
\begin{align*} |\inn{L_jf}{\psi}|&=|\inn{\widetilde L_jL_jf}{\widetilde L_j\psi}| \le \|\widetilde L_j L_j f\|_\infty \|\widetilde L_j\psi\|_1
\\
&\lc 2^{j\frac dp} \|L_j f\|_{L^{p,\infty}  } \min\{2^{-jN}, 2^{jM}\} 
\\&\lc  \|f\|_{\dot\cB_p^s(\ga,\infty)}
2^{j(\frac dp-s)}\min\{2^{-jN}, 2^{jM}\}
\end{align*}  where the implicit constants depend on $M,N,\psi$. Choosing $N$ large enough we see that $\sum_{j>0}|\inn{L_j f}{\psi}|<\infty$. Moreover 
$\sum_{j\le 0}|\inn{L_j f}{\psi}|<\infty$ if $M>s-d/p$ and thus $\sum_{j\in \bbZ} L_j f$ converges in $\cS_M'$.
\end{proof}

\begin{proof}[Proof of Proposition \ref{Prop::Diff::PropG}]
As already mentioned this could be skipped by citing  \cite{bourdaud-besov, madani}, which proves more, but in order to be self-contained we include a direct  proof of this somewhat easier result.
    We first define the  map
    \[\Biginclude_M:\cT_M\hookrightarrow \cS_{M}'(\bbR^d)\] as taking  \(\pi_Mu\) (with $u \in \cT$) to the distribution $\Biginclude_M\pi_Mu$ defined by 
    \[\inn{\Biginclude_M \pi_Mu}{\phi} = \int_{\R^d} u \phi\] and observe that $\Biginclude_M$ is injective. Also let $f \in \cup_{r \in [1,\infty]} \dot \cB^s_p(\ga,r) = \dot \cB^s_p(\ga,\infty)$.
By Lemma \ref{lem:LP-SMdistr}
    \(\sum_{k\in \Z} \Biginclude_M\pi_ML_kf\) converges in \(\cS_M'(\bbR^d)\) to some
    \(U\in \cSt_M'(\bbR^d)\). 
    
    We claim  \(U\in \Biginclude_M(\cT_M)\). To see this, decompose 
    \begin{equation*}
        U=U_{\high}+U_{\low}\coloneqq\sum_{k\geq 0} \Biginclude_M\pi_ML_k f +\sum_{k<0} \Biginclude_M\pi_ML_k f,
    \end{equation*}
    where the above sums converge in \(\cS_M'(\bbR^d)\).
    Since \(f\in\dot \cB^{s}_{p}(\gamma,\infty) \), we have \(\| L_k f\|_{L^{p,\infty}}\lesssim 2^{-ks}\),
    and since $s>0$ we see that  \(\sum_{k\geq 0} L_k f\) converges in \(L^{p,\infty}(\bbR^d)\) and 
    \begin{equation*}
        U_{\high}=\sum_{k\geq 0} \Biginclude_M\pi_M L_k f= \Biginclude_M  \pi_M \big[ \sum_{k\geq 0} L_k f \big]\in \Biginclude_M(\cT_M).
    \end{equation*}
    Since \(U_{\low}\in \cS_M'(\bbR^d)\), we can use  the Hahn-Banach Theorem to establish the existence of an extension  \(U_{\low}^{\mathrm{ext}}\in \cS'(\bbR^d)\)
    such that
    \begin{equation*}
        \inn{U_{\low}^{\mathrm{ext}}}{\psi}=\inn{U_{\low}}{\psi},\quad \forall \psi\in \cS_M(\bbR^d).
    \end{equation*}
    In particular, by the definition of \(U_{\low}\), we see that the Fourier transform of \(U_{\low}^{\mathrm{ext}}\) is supported in \(\{ |\xi|\leq 2\}\).
    Schwartz's Paley-Wiener Theorem implies there exists \(G\in \cT\) with \(\inn{U_{\low}^{\mathrm{ext}}}{\psi}=\int_{\R^d} G\psi\), for all \( \psi\in \cS(\R^d)\).
    It follows that \(U_{\low}=U_{\low}^{\mathrm{ext}}\big|_{\cS_M} = \Biginclude_M  \pi_M G\in \Biginclude_M(\cT_M)\).
    This completes the proof that \(U\in \Biginclude_M(\cT_M)\).
    
    We now can 
    define  \(\EMap f\);  because by injectivity of  $\Biginclude_M$  we have 
    \[U=\Biginclude_M(\EMap f),\]  for a unique \(\EMap f \in \cT_M\).  The map \(\EMap:f\mapsto \EMap f\) is then clearly linear.
    Also  \(U\big|_{\cS_\infty(\bbR^d)}=f\) and therefore \(\iota_M \EMap f=\Biginclude_M\EMap f\big|_{\cS_\infty(\bbR^d)}=U\big|_{\cS_\infty(\bbR^d)}=f\);
    that is, \(\iota_M\EMap\) is the identity.

We still need to establish the estimate
\Be \label{eq:BtofB}\| \EMap f \|_{\fB_{M,s,p}(\ga,r)}\lc \|f\|_{\dot \cB_p^s(\gamma,r)};\Ee
this is done using   the arguments in \S\ref{sec:norm-equivalence}. 
Define $g(x,k)\coloneqq2^{k(s+\frac{\ga}{p})} \EMap L_k f(x)$ so that $f = \sum_{k \in \Z} \cI_M 2^{-k(s+\frac{\ga}{p})} \widetilde L_k g(\cdot,k)$ with convergence in $\cS_M'$. By definition of $\dot \cB^s_p(\ga,r)$ we have $g \in L^{p,r}(\mu_{\ga})$. For all $h$ and a.e. $x$, 
\[
\frac{\Delta_h^M \EMap f(x)}{|h|^{s+\frac{\ga}{p}}} = T_{s+\frac{\ga}{p}} g(x,h)
\]
where $T_{s+\frac{\ga}{p}}$ is as in Lemma~\ref{lem:2.2}. Then we get \eqref{eq:BtofB} from Lemma~\ref{lem:2.2}.
\end{proof}

We turn to the map \(\fB_{M,s,p} (\ga,r)\rightarrow \dot\cB^s_p(\ga,r)\).

\begin{proposition}\label{Prop::Diff::PropF}
    For \(M\in \N\), \(s\in \R\), \(p\in (1,\infty)\), \(r\in [1,\infty]\), \(\gamma\in \R\), there is an injective bounded linear map
    \begin{equation*}
        \JMap: \fB_{M,s,p} (\ga,r)\rightarrow \dot\cB^s_p(\ga,r)
    \end{equation*}
    such that 
    \begin{equation}\label{Eqn::Diff::RestrictionEq}
        \JMap\big|_{\fB_{M,s,p} (\ga,r)\cap \cT_M}= \iota_M\big|_{\fB_{M,s,p} (\ga,r)\cap \cT_M}.
    \end{equation}
\end{proposition}
The main difficulty we must contend with in Proposition \ref{Prop::Diff::PropF} is that elements of \(\fB_{M,s,p} (\ga,r)\)
are a priori only equivalence classes of measurable functions (not necessarily locally integrable), and so we cannot  directly use any tools from distribution theory to study them. The following lemma appears to be well-known but we include a proof since we have not been  able to locate a precise reference.

\begin{lemma}\label{Lemma::Diff::AlmostPoly}
    Let $M\ge 1$, and  \(f:\bbR^d\rightarrow \bbC\) be  measurable  
    with  \(\Delta_h^M f(x)=0\) for $\mathscr {L}^{2d}$-almost every 
    \((h,x)\in \R^{2d}\).  
    Then there is a polynomial \(P\) of degree at most \(M-1\) such that \(f(x)=P(x)\)  almost everywhere.
\end{lemma}
Before proving the lemma we recall   some basic facts from  the theory of functional equations \cite{kuczma-book} which are needed in the proof.  
First,  we  need  a formula about iterated differences, attributed to Kemperman in \cite[Theorem 15.1.2]{kuczma-book}, see  also Djokovi\'c  \cite{djokovic} for related results. Namely, for all dimensions $d$, for all   $N\in \bbN$,   if $v^{(1)}$, ..., $v^{(N)} $  are vectors in $\bbR^d$ then
\Be\label{eq:iterated-differences}
\begin{split}&\Delta_{v^{(1)}}\dots\Delta _{v^{(N)}}f(x)= \sum_{(\ep_1,\dotsc,\ep_N)\in \{0,1\}^N}  (-1)^{\ep_1+\dots\ep_N}
\Delta_{h(\ep)}^M f(x+ \Tilde h(\ep)),
\\
&\text{ where } h(\ep)=- \sum_{j=1}^N j^{-1} \ep_jv^{(j)} ,\qquad 
\Tilde h(\ep)= \sum_{j=1}^N \ep_j v^{(j)} .
\end{split}
\Ee

Next we  recall that a $\mathscr{L}^d$-measurable function $f:\bbR^d\to \bbC$ is called {almost polynomial of order $M-1$} if $\Delta_h^Mf(x)=0$ for  $\mathscr{L}^{2d}$-a.e.\ $(x,h)\in \bbR^{2d}$.  
It is a result of Ger \cite{ger}, which we use in its form presented in \cite[Theorem 17.7.2]{kuczma-book},  that there exists a measurable function
 \(P \colon \bbR^d \to \bbC\) such that \(f(x)=P(x)\) for  $\mathscr{L}^d$-a.e.\ $x$ and $P$ is  a function satisfying 
    \(\Delta_h^M P(x) =0\), \emph{for all}  \( (h,x)\in \R^{d}\times \R^d\); such functions are called ``polynomial functions'' in \cite{ger}, \cite{kuczma-book}.
    
    We also use a result by Ciesielski \cite{ciesielski} (see  also \cite[Theorem 15.5.2]{kuczma-book}) 
    which states that if a measurable function $g:\bbR\to \bbC$  satisfies $\Delta_h^M g(x)\ge 0$ for all $x\in \bbR$ and all $h\in \bbR$ then $g$ is continuous; by an argument using weak derivatives this implies 
    that a polynomial function of order $M-1$ on the real line is actually a polynomial of degree at most $M-1$. In proving Lemma~\ref{Lemma::Diff::AlmostPoly} we could have used a $d$-dimensional version of this fact which could  be derived  from an abstract result by Kuczma \cite[Theorem 3]{kuczma}.
  However we  prefer to give a more direct argument 
  based on induction on $d$. 
   
   \begin{proof}[Proof of Lemma \ref{Lemma::Diff::AlmostPoly}] For $d=1$ Lemma~\ref{Lemma::Diff::AlmostPoly} is an immediate consequence of  the above mentioned theorems by Ciesielski and Ger. Let $d\ge 2$ and as induction hypothesis, assume Lemma~\ref{Lemma::Diff::AlmostPoly} in dimension $d-1$. 
We split variables as $x=(x',x_d)$.

Let $f:\bbR^d\to \bbC$  be almost  polynomial  of order $M-1$. By Ger's theorem there is  a measurable function  $g:\bbR^d\mapsto \bbC$ such that $f=g$ $\cL^d$-a.e.\ and $g$ is a polynomial function of order $M-1$. We therefore get $\Delta^M_{se_d} g(x)=0$ for all $x\in \bbR^d$ and all $s\in \bbR$. 
 Thus, for all   $x'\in \bbR^{d-1} $
we get from Ciesielski's theorem  that  the function $t\mapsto g(x',t)$ is a polynomial of degree at most $M-1$, i.e.\ we have 
\[g(x',x_d)= \sum_{j=0}^{M-1} a_j(x') x_d^j\]
for all  $x'\in \bbR^{d-1}$ and every $x_d\in \bbR$. 
The coefficient functions can be expressed via divided differences  in the $x_d$-variables (alternatively via derivatives)
and thus it is easily seen that each $a_j$ is $\mathscr{L}^{d-1}$-measurable. 
Since $\Delta_h^M g(x)=0$ for all $(x,h)$ we also have by
\eqref{eq:iterated-differences}
that $\Delta^{M-k}_{(u,0)}\Delta^k_{se_d} g(x)=0$, for all $x\in \bbR^d$, $u\in \bbR^{d-1} $ and  $s\in \bbR$. Letting $s\to 0$ (and using that $x_d\mapsto g(x',x_d)$ is polynomial) this implies that for $k=0,\dotsc,M$, 
\[ 0= \Delta^{M-k}_{(u,0)} \big(\frac{\partial }{\partial x_d})^k  g(x',x_d)= \sum_{j=k}^{M-1} \Delta_u^{M-k} a_j(x') c_{j,k}   x_d^{j-k},
\]
with $c_{j,k} =\prod_{i=1}^k (j-i+1).$
This  in turn implies 
$\Delta_u^{M-k} a_k(x')=0$ for $k=0,\dotsc, M$, and all $u\in \bbR^{d-1}$. Thus, by the induction hypothesis $a_k(x')$ is almost everywhere equal to a polynomial of degree at most $M-k-1$, and we deduce that $g$ and thus $f$ is $\mathscr {L}^d$-a.e.\ equal to a polynomial of degree at most $M-1$.
\end{proof}

\begin{lemma}\label{Lemma::Diff::IntLpr}
    Fix \(\gamma\in \R\), \(p\in (1,\infty)\), \(r\in [1,\infty]\).  Then, if \(K,L\in \N\) are sufficiently large, we have
    \begin{equation*}
        \iint | F(x,h) | \min \{  |h|^K, |h|^{-K} \} ( 1+|x| )^{-L} \dif x \dif h \lesssim \| F \|_{L^{p,r}(\nu_\gamma)},
    \end{equation*}
for all $F\in L^{p,r}(\nu_\gamma)$.
\end{lemma}
\begin{proof}
   By interpolation it  suffices to show this for $r=p$ (possibly, after increasing $K$). 
    The desired bound follows if we can show  \(K_1,L\in \N\) sufficiently large, that
    \begin{equation*}\label{Eqn::Diff::ToShowInDualLpr}
        \min \{  |h|^{K_1}, |h|^{-K_1} \} ( 1+|x| )^{-L}\in L^{p'}(\nu_\gamma),
    \end{equation*}
    where $p'=\frac{p}{p-1}$. This however is elementary.
\end{proof}

Before we define the operator \(\JMap\) from Proposition \ref{Prop::Diff::PropF}, we introduce some auxiliary operators.
Let \(j\in \Z\), \(m\in C_0^{\infty}(\bbR^d\setminus\{0\})\), \(w\in \bbR^{d}\setminus\{0\}\), and \(\eps>0\)
be such that \(\overline{\Ball{w}{\eps}}\subset \{\xi:1/2<|\xi|<2\} \).
For \(\psi\in \cS_\infty(\bbR^d)\), define
\(\Gamma^j_{m,w,\eps}\psi(x,h)\) by
\begin{equation}\label{Eqn::Diff::DefnGamma}
    \big[\Gamma^j_{m,w,\epsilon} \psi  \big]^{\wedge} (\xi, h)\coloneqq 2^{jd} m(-2^{-j}\xi) \widehat{\psi}(\xi) \, \Indicator_{\Ball{2^{-j}w}{2^{-j}\epsilon}}(h),
\end{equation}
where \(\wedge\) denotes the Fourier transform in the \(x\rightarrow \xi\) variable.

\begin{lemma}\label{Lemma::Diff::BoundGamma}
    Let \(\Omega\subset \cS_\infty(\bbR^d)\) be a bounded set.
    Then for all \( K,L\in \N\) and \(\alpha\in \N^{d}\), there exists \(C_{K,L,\alpha,\Omega}\geq 0\), which may depend on the fixed $j,m,w$ and $\epsilon$,
    such that
    \begin{equation*}
        | \partial_x^{\alpha} \Gamma^j_{m,w,\eps} \psi (x,h) |
        \leq C_{K,L, \alpha,\Omega} 2^{-|j|}\min \{ |h|^K, |h|^{-K} \} ( 1+|x| )^{-L} 
    \end{equation*}
    for all \( \psi\in \Omega\).
\end{lemma}
\begin{proof}
    Equivalently, we wish to show that the 
    set 
    \begin{equation*}
        \Big\{  \max\{ |h|^K, |h|^{-K} \}  2^{|j|}\Gamma^j_{m,w,\eps} \psi(\cdot, h) : h\in \R^{d}\setminus\{0\}, \psi\in \Omega  \Big\}
    \end{equation*} 
    is bounded in $\cS(\bbR^d)$.
    Since \(\Gamma^j_{m,w,\epsilon} \psi(x, h)=0\) unless \(|h|\approx 2^{-j}\), it suffices to show that
    \begin{equation*}
        \big\{  2^{|j|(K+1)}  \Gamma^{j}_{m,w,\epsilon} \psi(\cdot, h) :  \psi\in \Omega, h\in \R^d\setminus \{0\}  \big\}
    \end{equation*}
    is  bounded in $\cS(\bbR^d)$. 
    Taking the Fourier transform, this follows if we show that 
    \begin{equation*}
        \big\{  2^{|j|(K+1)}  2^{jd} m(-2^{-j}\xi) \widehat{\psi}(\xi) : \psi\in \Omega  \big\} 
    \end{equation*} is bounded in $\cS(\R^d)$.
    Using that \(\supp\{ m(2^{-j}\cdot)\}\subset \{|\xi|\approx 2^j\}\), for \(m\in C_0^{\infty}(\bbR^{d}\setminus\{0\})\),
    and \(\Omega\subset \cS_\infty(\bbR^d)\) is a bounded set, this follows, completing the proof.
\end{proof}

For \(b\in \R\), \(p\in (1,\infty)\), \(r\in [1,\infty]\), \(\gamma\in \R\), \(F\in L^{p,r}(\nu_\gamma)\), and \(\psi\in \cS_\infty(\bbR^d)\), set
\begin{subequations}\label{eq:defU}
\begin{align}\label{Eqn::Diff::DefineTj}
    \inn{U^{b,j}_{m,w,\eps }F}{\psi}&\coloneqq 
    \iint |h|^b F(x,h) \, \Gamma^j_{m,w,\eps}\psi(x,h)\dif x \dif h\\ \label{Eqn::Diff::DefineT}
     \inn{U^b_{m,w,\eps }F}{\psi}&\coloneqq 
    \sum_{j\in \Z}  \inn{U^{b,j}_{m,w,\eps }F}{\psi}.
    \end{align}
\end{subequations}
\begin{lemma}\label{Lemma::Diff::Converge}
    For $F\in L^{p,r}(\nu_\ga)$, the sums and integrals in \eqref{eq:defU} converge absolutely and \eqref{Eqn::Diff::DefineT} defines
    \(U^b_{m,w,\eps }F\in \cS_\infty'(\bbR^d)\).
\end{lemma}
\begin{proof} 
    By Lemmas \ref{Lemma::Diff::BoundGamma} and \ref{Lemma::Diff::IntLpr}, we have for any \(K,L\in \N\) sufficiently large,
    \begin{equation*}
    \begin{split}
         &\sum_{j\in \Z} \iint |h|^b |F(x,h)| | \Gamma^{j}_{m,w,\eps}\psi(x,h) | \dif x \dif h
         \\&\lesssim_{K,L} \sum_{j\in \Z} 2^{-|j|} \iint |F(x,h)|  \min \{ |h|^K, |h|^{-K} \} ( 1+|x| )^{-L} \dif x \dif h
         \\&\lesssim \sum_{j\in \Z} 2^{-|j|} \| F\|_{L^{p,r}(\nu_\gamma)}
         \lesssim \| F\|_{L^{p,r}(\nu_\gamma)}.
    \end{split}
    \end{equation*}
    This shows the absolute convergence and defines \(U^b_{m,w,\eps }F\) in the algebraic dual of
    \(\cS_\infty(\bbR^d)\).

    To see that \(U^b_{m,w,\eps }F\in \cS_\infty'(\bbR^d)\), let \(\psi_k\in \cS_\infty(\bbR^d)\) be such that
    \(\psi_k\rightarrow \psi\) in \(\cS_\infty(\bbR^d)\).
    In particular, \(\{ \psi_k: k\in \N\} \)\ is a bounded set in \(\cS_\infty(\R^d)\) and therefore by 
    Lemma \ref{Lemma::Diff::BoundGamma},
    \begin{equation*}
        |h|^b \mleft| \Gamma^{j}_{m,w,\eps}\psi_k(x,h) \mright| \lesssim 2^{-|j|}\min \mleft\{ |h|^K, |h|^{-K} \mright\} \mleft( 1+|x| \mright)^{-L},
    \end{equation*}
    with implicit constant independent of \(k\).  Combining this with Lemma \ref{Lemma::Diff::IntLpr},
    the dominated convergence theorem shows \(\inn{U^b_{m,w,\eps}F}{\psi_k}\rightarrow \inn{U^b_{m,w,\eps}F}{\psi}\),
    completing the proof.
\end{proof}

\begin{lemma}\label{Lemma::Diff::BoundT}
    For \(b,\gamma\in \R\), \(p\in (1,\infty)\),  \(r\in [1,\infty]\), 
    \[U^b_{m,w,\eps}:L^{p,r}(\nu_\gamma)\rightarrow \dot\cB^{b-\gamma/p}_{p}(\ga,r)\]  is a bounded linear transformation.
\end{lemma}
\begin{proof} This is  an application of Lemma  \ref{lem:V-lemma}.
From the definitions 
\eqref{Eqn::Diff::DefnGamma} and \eqref{Eqn::Diff::DefineTj} we get 
\begin{align*}
L_k U^{b,j}_{m,w,\eps} F(x) &= 2^{-jb} \int\limits_{|h-2^{-j}w| \leq \eps 2^{-j}} (2^j|h|)^b \varphi(2^{-k}D) m(2^{-j}D) [F(\cdot,h)] (x) \frac{\dif h}{2^{-jd}} 
\end{align*} 
and thus $L_k U^{b,j}_{m,w,\eps}=0$ when $|k-j|\ge 2$.
Then with $V^{b,j}_{m,w,\eps}$
as in \eqref{eq:Vbka-def}, we get for $n=-1,0,1$, 
\Be \label{eq:UversusV}
L_k U^{b, k+n}_{m,w,\eps} F
= 2^{-(k+n) b} V^{b,k+n}_{\widetilde m_n,w,\eps} F, \,\,\text{ with } \widetilde m_n= \varphi (2^{n}\cdot ) m, 
\Ee
and
\[ P^b U^{b}_{m,w,\eps} F(\cdot,k) = 
2^{kb} L_k U^{b}_{m,w,\eps} F = 
\sum_{n=-1}^{1} 2^{-nb}V^{b,k+n}_{\widetilde m_n,w,\eps} F. \] Hence 
\begin{align*}  \| U^b_{m,w,\eps} F\|_{\dot \cB^{b-\ga/p}_p(\ga,r) } &\le \sum_{n=-1}^1 2^{-nb} 
\|V^b_{\widetilde m_n,w,\eps}F (\cdot, \cdot+n)\|_{L^{p,r}(\mu_\ga)}
\\
&\lc_{b,\ga} \sum_{n=-1}^1 \|V^b_{\widetilde m_n,w,\eps}F\|_{L^{p,r}(\mu_\ga)}
\end{align*}
and since by Lemma \ref{lem:V-lemma} we have 
\(\|V^b_{\widetilde m_n,w,\eps}F \|_{L^{p,r}(\mu_\ga)} \lc 
\|F\|_{L^{p,r} (\nu_\gamma)}\) the proof is complete.
\end{proof}

 The following  lemma has  a dual version   of formula \eqref{eq-resolution-of-f} and an extension to tempered functions. 
\begin{lemma}\label{Lemma::Diff::IdentOnsT} Let $\psi\in \cS_\infty(\bbR^d)$. Then 
\begin{equation}\label{Eqn::Diff::SumToGetLPScale}
        \sum_{\kappa=1}^N  \int \Delta_{-h}^M  \Gamma^{j}_{m_{\kappa}, w_{\kappa}, \eps} 
        \psi (x,h)\dif h = L_j \psi(x).
    \end{equation}
    Moreover, for \(f\in \cT\) and \(\psi\in \cS_\infty(\bbR^d)\),
    \begin{equation}\label{eq:identity-expansion}
        \sum_{\kappa=1}^N \sum_{j\in \Z}  \iint  \Delta_{h}^M f (x) \, \Gamma^{j}_{m_{\kappa}, w_{\kappa},\eps} \psi(x,h)\,\dif h \dif x = \int f(x) \psi(x)\: \dif x.
    \end{equation}
\end{lemma}
\begin{proof} We first check  \eqref{Eqn::Diff::SumToGetLPScale},
which, after taking the Fourier transform,  is equivalent with 
\begin{equation}\label{Eqn::Diff::SumToGetLPScale::ToShow}
        \sum_{\kappa=1}^N  \int ( e^{i \inn{\xi}{-h}}-1) ^M \mleft[ \Gamma^{j}_{m_{\kappa}, w_{\kappa},\eps} \psi \mright]^{\wedge}(\xi,h)\: \dif h = \varphi(2^{-j} \xi)\widehat{\psi}(\xi).
    \end{equation}
    Using \eqref{Eqn::Diff::DefnGamma} and 
    \eqref{eq:decomp-of-phi} we have
    \begin{equation*}
    \begin{split}
         &\sum_{\kappa=1}^N \int \big( e^{i \inn{\xi}{-h}}-1 \big)^M \mleft[ \Gamma^j_{m_{\kappa}, w_{\kappa}, \eps}\psi \mright]^{\wedge}(\xi,h)\: \dif h
         \\&= \sum_{\kappa=1}^N 2^{jd} \int_{|h-2^{-j}w_{\kappa}|<2^{-j}\eps} ( e^{i \inn{-\xi}{h}}-1 )^M m_{\kappa}(-2^{-j}\xi) \widehat{\psi}(\xi)\dif h
         \\&=\varphi(-2^{-j} \xi) \widehat{\psi}(\xi) =\varphi(2^{-j} \xi)\widehat \psi(\xi),
    \end{split}
    \end{equation*}
    here we used that $\varphi$ is radial. This establishes  \eqref{Eqn::Diff::SumToGetLPScale::ToShow} and thus 
\eqref{Eqn::Diff::SumToGetLPScale}.

We now prove \eqref{eq:identity-expansion}. In the argument that follows  all integrals and sums converge absolutely by Lemmas \ref{Lemma::Diff::IntLpr} and  \ref{Lemma::Diff::BoundGamma}.
    Using \eqref{Eqn::Diff::SumToGetLPScale} we have 
    \begin{equation*}
    \begin{split}
         &\sum_{\kappa=1}^N \sum_{j\in \Z}  \iint  \Delta_{h}^M f (x) \, \Gamma^{j}_{m_{\kappa}, w_{\kappa},\epsilon} \psi (x,h)\dif h \dif x \\&=\sum_{\kappa=1}^N \sum_{j\in \Z}  \int f(x) \int \Delta_{-h}^M \Gamma^j_{m_{\kappa}, w_{\kappa}, \epsilon}\psi (x,h)\dif h \dif x
         \\&= \sum_{j\in \Z} \int f(x) \, L_j \psi (x) \dif x
         =\int f(x)\psi(x)\dif x,
    \end{split}
    \end{equation*}
    where the final equality uses $f \in \cT$ and that \(\sum_{j\in \Z} L_j \psi =\psi\), with convergence in \(\cS_\infty(\bbR^d)\), since \(\psi\in \cS_\infty(\bbR^d)\).
\end{proof}

We are prepared to define \(\JMap\). For $f\in \cM$ with $\pi_M f \in \fB_{M,s,p} (\ga,r)$ we set 
\begin{subequations}
\begin{align}\label{eq:JM-dualdef} 
\inn{\mathscr J_M(\pi_Mf) }{\psi} &\coloneqq \sum_{\ka=1}^N \sum_{j\in \bbZ} \iint \Delta_h^M f(x) \Gamma^j_{m_\ka,w_\ka,\eps} \psi(x,h) \dif x\dif h
\\ \label{Eqn::Diff::JDefn}
&= \sum_{\ka=1}^N \inn{U^b_{m_\ka,w_\ka,\eps} F_b}{\psi} \text{ with } F_b(x,h)=\frac{\Delta^M_h f(x)}{|h|^b}
\end{align}
\end{subequations}
where by Lemma \ref{Lemma::Diff::Converge} the sums and integrals in \eqref{eq:JM-dualdef} 
converge absolutely. 
Note that the definition of $\mathscr J_M$ depends on $M$, but not on  $s,\ga,p, r$, and that \eqref{Eqn::Diff::JDefn} holds for all $b\in \bbR$. We shall later use this formula with $b=s+\ga/p$. 
When \(f\in \cT\),  Lemma 
\ref{Lemma::Diff::IdentOnsT} shows that 
$\inn{\mathscr J_M(\pi_Mf) }{\psi}$ is the standard pairing of $f\in \cT$ with a Schwartz function in $\cS_\infty$, i.e.
\begin{equation} \label{eq:Tid} \inn{\mathscr J_M(\pi_Mf) }{\psi}
= \int f(x)\psi(x) \dif x, \quad \forall f \in \cT.\end{equation}

We need to show  that 
 \(\JMap\) is injective on \(\fB_{M,s,p} (\ga,r)\). For this, we will need the following auxiliary lemma.
\begin{lemma}\label{zeropol}
Let $p\in (1,\infty)$, $r\in [1,\infty]$, and $\gamma\in \bbR$.
Suppose that $F\in L^{p,r}(\nu_\ga)$ and $\eta\in C^\infty_c(\bbR^d\setminus \{0\})$ are such that
\[x\mapsto Q(x)\coloneqq\int F(x,h) \eta(h) \dif h\]
is almost everywhere equal to a polynomial. 
Then $Q(x)=0$ almost everywhere.
\end{lemma}
\begin{proof} Let $\phi \in C^\infty_c(\bbR^d)$ be nonnegative and  $\int\phi=1$. We claim that, for all $G\in L^{p,r}(\nu_\ga)$,
\Be \label{eq:atoinfty}
\lim_{|a|\to\infty} \iint G(x,h) \eta(h) \phi(x-a) \dif h\dif x  =0
\Ee  
Observe that \eqref{eq:atoinfty} follows by standard estimates whenever $G\in L^q(\nu_\gamma)$ for any $q\in (1,\infty)$. It then also holds for  $G\in L^{p,r}(\nu_\gamma)$ since 
$L^{p,r}(\nu_\ga) \subset L^{p_1}(\nu_\ga)+L^{p_2}(\nu_\ga)$, with $p_1<p<p_2$.

By \eqref{eq:atoinfty} we have
\begin{align*}
    0&= \varlimsup_{|a|\to \infty}\Big|
    \iint F(x,h) \eta(h) \phi(x-a) \dif h\dif x \Big|
    = \varlimsup_{|a|\to \infty}\Big|
    \int Q(x) \phi(x-a) \dif x \Big|
\end{align*}
and the last expression is equal to $|c|$ if $Q(x)=c$ almost everywhere, 
and equal to $\infty$ if $Q$ is almost everywhere   equal to a nonconstant polynomial. We conclude that $Q(x)=0 $ almost everywhere.
\end{proof}
\begin{lemma}\label{Lemma::Diff::JInjective}
    For \(M\in \N\), \(s,\gamma\in \R\), \(p\in (1,\infty)\), and \(r\in [1,\infty]\),
    \(\JMap\) is injective on \(\fB_{M,s,p} (\ga,r)\).
\end{lemma}
\begin{proof}
    Suppose $f \in \cM$ is such that \(\pi_Mf\in \fB_{M,s,p} (\ga,r)\) and \(\JMap\pi_M f=0\)
    as an element of \(\cS_\infty'(\bbR^d)\).  We wish to show \(f(x)=P(x)\), almost everywhere,
    for some polynomial \(P(x)\) of degree \(\leq M-1\).
    In this proof, all sums and integrals converge absolutely by Lemmas
    \ref{Lemma::Diff::IntLpr} and \ref{Lemma::Diff::BoundGamma}.

    Take \(\psi\in \cS_\infty(\bbR^d)\) and \(\eta\in C_c^\infty(\bbR^d\setminus\{0\})\).
    Then,
    \begin{equation*}
        \int \eta(h') \Delta_{h'}^M \psi(x)\dif h' \in \cS_\infty(\bbR^d).
    \end{equation*}
    Thus, we have, using \eqref{eq:JM-dualdef} and the definition  of the translation invariant operator  \(\Gamma^{j}_{m_{\kappa}, w_{\kappa}, \epsilon}\)  (see \eqref{Eqn::Diff::DefnGamma}),
    \begin{equation*}
    \begin{split}
         &0=\Big \langle\JMap\pi_M f, \int \eta(h') \, \Delta_{-h'}^M \psi \dif h'\Big \rangle
         \\&=\sum_{\kappa=1}^N \sum_{j\in \Z} \iiint  \Delta_{h}^M f (x) \, \eta(h') \, \Delta_{-h'}^M \Gamma^j_{m_{\kappa},w_{\kappa},\eps} \psi (x,h)\dif x \dif h \dif h'
         \\&=\sum_{j\in \Z} \iint  \Delta_{h'}^M f (x) \eta(h') \sum_{\kappa=1}^N \int   \Delta_{-h}^M \Gamma^{j}_{m_{\kappa}, w_{\kappa}, \eps} \psi (x,h)\dif h\dif x\dif h'
         \\&=\sum_{j\in \Z}\iint \Delta_{h'}^M f (x) \eta(h') \dif h' \, L_j\psi (x) \dif x,
    \end{split}
    \end{equation*}
    where the last equality uses \eqref{Eqn::Diff::SumToGetLPScale}.
    It follows from Lemma \ref{Lemma::Diff::IntLpr} and the fact that \(\eta\in C_c^\infty(\bbR^d\setminus\{0\})\)
    that \(\int \Delta_{h'}^M f (\cdot)  \eta(h') \dif h' \in \cT\).
    Since \(\psi\in \cS_\infty(\bbR^d)\), we have \(\sum_{j\in \Z} L_j \psi =\psi\) with convergence in \(\cS_\infty(\bbR^d)\).
    Thus,
    \begin{equation}\label{Eqn::Diff::First0Int}
        \begin{split}
            &0=\sum_{j\in \Z}\iint  \Delta_{h'}^M f (x) \eta(h') \dif h'  L_j\psi (x) \dif x
            \\&=\iint  \Delta_{h'}^M f( x) \,\eta(h') \dif h' \psi(x) \dif x, 
        \end{split}
    \end{equation}
for arbitrary \(\psi\in \cS_\infty(\bbR^d)\) and  we can  conclude that \[\int  \Delta_{h'}^M f (x)\eta(h')\dif h'=Q(x), \text{ a.e. } 
    \] 
    for some polynomial \(Q\).  By Lemma \ref{zeropol} it follows that $Q=0$, hence  \[\int  \Delta_{h'}^M f (x)\eta(h')\dif h'=0, \text{  a.e. }\]
    Since \(\eta\in C_0^\infty(\bbR^{d}\setminus\{0\})\) was arbitrary, this implies
    \(\Delta_{h}^M f (x) =0\) for almost every \((x,h)\in \bbR^{2d}\).
    Lemma \ref{Lemma::Diff::AlmostPoly} shows \(f(x)=P(x)\), almost everywhere, for some polynomial \(P(x)\)
    of degree \(\leq M-1\), completing the proof.
\end{proof}

\begin{proof}[Proof of Proposition \ref{Prop::Diff::PropF}]
    It follows immediately from the definitions that
    \begin{equation*}\begin{cases}
        &\pi_Mf\mapsto \Big( (x,h)\mapsto \frac{\Delta_{h}^M f(x)}{|h|^{s+\gamma/p}}\Big) 
        \\
        &\fB_{M,s,p} (\ga,r)\rightarrow L^{p,r}(\nu_\gamma)
        \end{cases}
    \end{equation*}
    is  bounded. 
    Lemma \ref{Lemma::Diff::BoundT} shows that
    \(U_{m_{\kappa}, w_{\kappa}, \eps}^{ s+\gamma/p} : L^{p,r}(\nu_\gamma)\rightarrow \dot\cB^s_p(\ga,r)\)
    is bounded.  Composing these maps and using \eqref{Eqn::Diff::JDefn} shows that 
    \[\JMap:\fB_{M,s,p} (\ga,r)\rightarrow \dot\cB^s_p(\ga,r) \] is bounded.
    \(\JMap\) is injective by  Lemma \ref{Lemma::Diff::JInjective}.
    Finally, \eqref{Eqn::Diff::RestrictionEq} follows from \eqref{eq:Tid}. 
\end{proof}

\begin{proof}[Proof of Theorem \ref{Thm::Diff::MainDiffThm}, conclusion]
    By Proposition \ref{Prop::Diff::PropG}, \[\EMap(\dot\cB^s_p(\ga,r))\subseteq \cT_M\cap \fB_{M,s,p} (\ga,r),\]
    and so by \eqref{Eqn::Diff::RestrictionEq}, \(\JMap\big|_{\EMap(\dot\cB^s_p(\ga,r))} = \iota_M\big|_{\EMap(\dot\cB^s_p(\ga,r))}\).
    By Proposition \ref{Prop::Diff::PropG}, \(\iota_M\) is a left inverse to \(\EMap\),
    and we conclude \(\JMap\EMap \) is the identity map on  \(\dot\cB^s_p(\ga,r)\).
    In particular, \(\JMap\big|_{\cT_M\cap \fB_{M,s,p} (\ga,r)}:\cT_M\cap \fB_{M,s,p} (\ga,r) \rightarrow  \dot\cB^s_p(\ga,r)\)
    is surjective.
    Proposition \ref{Prop::Diff::PropF} shows \(\JMap: \fB_{M,s,p}(\ga, r)\rightarrow \dot\cB^s_p(\ga,r)\) is injective.
    We conclude \Be \label{eq:locally-integrable}\cT_M\cap \fB_{M,s,p} (\ga,r)=\fB_{M,s,p}(\ga,r),\Ee establishing part  \ref{Item::Diff::IncTM} of the theorem,
    and moreover that 
    \(\JMap:\fB_{M,s,p}(\ga, r)\rightarrow \dot\cB^s_p(\ga,r)\) is bijective with
    two-sided inverse \(\EMap\).
From \eqref{eq:locally-integrable}
    and \eqref{Eqn::Diff::RestrictionEq}
    we see that  \[\JMap:\fB_{M,s,p}(\ga, r)\rightarrow \dot\cB^s_p(\ga,r)\] agrees with \(\iota_M\) on all of \(\fB_{M,s,p} (\ga,r)\).
    Thus, \[ \iota_M\big|_{\fB_{M,s,p} (\ga,r)} : \fB_{M,s,p} (\ga,r) \rightarrow \dot\cB^s_p(\ga,r)\]
    is a bounded bijective map with bounded inverse \(\EMap\).  This establishes parts \ref{Item::Diff::InIota} and \ref{Item::Diff::Bijection} of the theorem,
    completing the proof.
\end{proof}

\section{Embeddings} \label{sec:embeddings}
The proof of the embeddings in Theorem \ref{thm:embeddings} is reduced to inequalities for the operator $T_a$
defined  on functions $F: \bbR^d\times \bbZ\to \bbC$  by
\Be\label{Ta} T_a  F(x,j) = 2^{ja} F_j(x),\Ee
with  the parameters  $a=\pm\ga/p$.

\begin{lemma} \label{lem:cB-B-rev-emb}
The following hold for all  $\ga\in \bbR$, $1<p<\infty$.

(i) For $p\le r\le \infty  $, 
\[ \|T_{-\ga/p}  G\|_{\ell^r(L^{p,r})} \lc \|G\|_{L^{p,r} (\mu_\ga)}.
\]

(ii) For $1\le r\le p$
\[
\|T_{\ga/p}  F\|_{L^{p,r} (\mu_\ga)}\lc \|F\|_{\ell^r(L^{p,r})}.\]
\end{lemma}

\begin{proof} 
 Part (i)  follows from the definitions of Lorentz spaces via the distribution function. We use a change of variable with subsequent interchange of sum and integral to write
 \begin{align*}
 \|T_{-\ga/p} G\|_{\ell^r(L^{p,r})} &\lc \Big(\sum_j \int_0^\infty \la^r \big[\meas\{x:\, 2^{-j\frac {\ga}{p}}|G(x,j)|>\la\}\big]^{r/p} \frac{\dif \la}{\la}\Big)^{1/r}
 \\
 &= \Big(\int_0^\infty \beta^r \sum_j \big [2^{- j\ga } 
 \meas\{ x: |G(x,j)|>\beta\} \big]^{r/p } \frac{\dif \beta}{\beta} \Big)^{1/r}
 \end{align*} and since $r\ge p$  we  estimate an $\ell^{r/p}$-norm by an $\ell^1$-norm and see that the last displayed expression  is dominated by 
 \begin{align*}
&\Big(\int_0^\infty \beta^r \Big[ \sum_j 2^{-j\ga }
 \meas\{ x: \,|G(x,j)|>\beta\}\Big ]^{r/p} \frac{\dif \beta}{\beta}  \Big)^{1/r}
\\&= \Big(\int_0^\infty \beta^r \big[\mu_\ga 
\{(x,j):\,|G(x,j)|>\beta\}\big]^{r/p} \frac {\dif \beta}{\beta} 
\Big)^{\frac 1r} \lc \|G\|_{L^{p,r}(\mu_\ga)} . 
\end{align*}

For part (ii) we use that $(L^{p,r}(\mu_\ga))^*= L^{p',r'}(\mu_\ga)$, (\cite{hunt})  and $(\ell^r(L^{p,r}))^*= \ell^{r'}(L^{p',r'})   $.
Observe  that for $1\le r\le p$ 
\begin{align*}
    &\Big|\int \sum_j T_{\ga/p}F(x,j)  G(x,j) \frac{\dif x}{2^{j\gamma}}\Big|= 
    \Big| \sum_j \int F(x,j) 2^{-j\ga/p'} G(x,j) \dif x \Big|
    \\&\lc \|F\|_{\ell^r(L^{p,r})} \|T_{-\ga/p'} G\|_{\ell^{r'} (L^{p',r'})} \lc  
     \|F\|_{\ell^r(L^{p,r})} \| G\|_{ L^{p',r'} (\mu_\gamma) } 
    \end{align*}
    where we have used part (i) for the exponents $p'\le r'$. The proof is completed by taking the supremum over all $G$ with 
    $\| G\|_{ L^{p',r'} (\mu_\gamma) } \le 1$.
\end{proof}

\begin{lemma} \label{lem:cB-F-rev-emb}
Let  $1<p<\infty$,   $\ga \neq 0$.

(i) For $p\le r\le \infty$,
\[\|T_{\ga/p} F\|_{L^{p, r} (\mu_\ga )} \lc \|F\|_{L^p(\ell^r)}.\]

(ii)  For $1\le r\le p$,
\[ \|T_{-\ga/p} G\|_{L^{p}(\ell^r)} \lc \|G\|_{L^{p, r} (\mu_\ga )}\]
\end{lemma}

\begin{proof} The argument for part (i) has been used in proofs for endpoint multiplier theorems, our proof is essential the  one from \cite[Lemma 2.4]{lrs-stein} (see also \cite{lrs-stein} for further references).

Let $0\leq \theta \leq 1$ and $1/r=(1-\theta)/p$. We use the complex interpolation formulas
\begin{align*}
[ L^p(\ell^p), L^p(\ell^\infty)]_\theta= L^p(\ell^r),
\quad 
[L^p(\mu_\gamma), L^{p,\infty}(\mu_\gamma) ]_\theta= L^{p,r}(\mu_\gamma) .
\end{align*}
These imply that it suffices to prove the assertion for $r=p$ and $r=\infty$.
For $r=p$ we have
$\|T_{\ga/p}  F\|_{L^{p} (\mu_\ga)} = \|F\|_{L^p(\ell^p)}$.

For $r=\infty$ the  conclusion $T_{\ga/p} : L^p(\ell^\infty)\to L^{p,\infty} (\mu_\ga)$ follows from \begin{align*}
&\mu_\ga\{(x,j): |T_{\ga/p} F(x,j)|>\la\} = \int_{\R^d} \sum_{\substack{j : 
\\2^{j\ga/p }|F_j(x)|>\la}}  2^{-j\ga} \dif x 
\\
&\le \int_{\R^d} \sum_{\substack{j : \\2^{-j\ga/p} <\la^{-1} \sup_k |F_k(x)|} }  2^{-j\ga} \dif x  
\lc \int_{\R^d} \frac{(\sup_k|F_k(x)|)^p} {\la^p} \dif x;
\end{align*} here we used $\gamma \neq 0$.

For part (ii) we use that $(L^{p,r}(\mu_\ga))^*= L^{p',r'}(\mu_\ga)$, see \cite{hunt} and $(L^p(\ell^r))^*= L^{p'}(\ell^{r'})$.
Observe that for $F\in L^{p'}(\ell^{r'})$
\begin{align*}
   & \int \sum_j 2^{-j\ga/p}G_j(x) F_j(x)\dif x=\int \sum_j  G_j(x)\, T_{\ga/p'}F(x,j)  2^{-j\gamma} \dif x 
   \\&\lc \|G\|_{L^{p,r}(\mu_\gamma)}  \|T_{\ga/p'} F\|_{L^{p',r'} (\mu_\gamma)}  \lc 
   \|G\|_{L^{p,r}(\mu_\gamma)}  \|F\|_{L^{p'}(\ell^{r'} )} 
   \end{align*}
where we have used part (i). Now part (ii) follows by taking the $\sup$ over all $F$ with  $\|F\|_{L^{p'}(\ell^{r'} )} \le 1$.
\end{proof} 
\begin{proof}[Proof of Theorem \ref{thm:embeddings}]
Apply Lemma \ref{lem:cB-B-rev-emb} and Lemma \ref{lem:cB-F-rev-emb}
with $F(x,j)=2^{js} L_j f(x)$ and $G(x,j)=2^{j(s+\frac \ga p)} L_j f(x)$. 
\end{proof}

\section{Non-embeddings}\label{sec:non-embeddings}
We prove Theorem \ref{thm:non-emb}. Proposition \ref{prop:counterex-ga>-d} covers part (i) and (ii) of the theorem, in the range $\ga\ge -d$, and Proposition \ref{prop:counterex-ga<-d} covers the same parts for the range $\ga<-d$. Proposition \ref{prop:F-B} covers part (iii) of Theorem \ref{thm:non-emb}. We begin with some definitions to build the examples.

If $\ga \ge -d$ and $k > 0$, or if $\ga < -d$ and $k < 0$, define
\[ 
\fN_\ga(k) \coloneqq \lfloor 2^{k (d+\gamma)} \rfloor. 
\]
Let $\{n_{i,k} \}$ be a double indexed set in $\bbZ$, with $1\le i\le \fN_\gamma(k)$,   which is  separated in the sense that for every $k$
\[ i_1\neq i_2 
\implies |n_{i_1,k}-n_{i_2,k} | \ge 
2^ {10|k|} .
\]

Let $\eta \in \cS$ such that \begin{subequations}
\begin{align} \label{eq:etaapprox}
    &|\eta(x)|\approx 1 \text{  for $|x|\le 1$.}
    \\ \label{eq:etaFTsupport}
    &\supp(\widehat \eta) \subset \big\{\xi\in \widehat \bbR^d: \frac{15}{16}\le |\xi|\le \frac{17}{16} \big\}
\end{align} 
\end{subequations} and let
\Be\eta_{i,k}(x)= \eta(2^k (x-n_{i,k}e_1)).\Ee
By \eqref{eq:varphiassumptions} we have 
\begin{subequations}
\Be \label{eq:LP-repr} 
\eta_{i,k}=L_k\eta_{i,k}
\Ee and 
\Be 
L_\ell \eta_{i,k} = 0 \text {  if  } \ell \ne k.
\Ee
\end{subequations}

Define for $k\in \bbZ$ 
\Be \label{eq:fgak} f_{\ga,k} (x) = 2^{-k\ga/p} \sum_{i=1}^{\fN_\ga(k) }\eta_{i,k} .
\Ee 
\begin{proposition} \label{prop:counterex-ga>-d} Let $f_{\ga,k}$ be as in \eqref{eq:fgak}. Let $s\in \bbR$.
Assume $\gamma\ge -d$, and define\footnote{The definitions in \eqref{eq:FgaN-kpos}, \eqref{eq:FG-defs} depend on $s$ but we do not include the subscript $s$ to keep the notation as simple as possible.} 
\Be \label{eq:FgaN-kpos} 
F_{ \ga,N} (x) = 
 \sum_{k=N+1}^{2N} 2^{-ks} f_{\ga,k}(x) .
\Ee 
(i) Then 
$1<p\le r\le\infty$
\begin{align}
\label{eq:Bspgarbelow}
&\|F_{\ga, N} \|_{\dot \cB^s_p(\gamma,\infty)}  =\|F_{\ga,N} \|_{ \cB^s_p (\gamma,\infty)}
\gc N^{1/p}
\\ \label{eq:Bspbetarabove} &\|F_{\ga,N} \|_{\dot \cB^s_p (\beta,r)}  =\|F_{\ga,N} \|_{ \cB^s_p (\beta,r)}  
        \lc N^{1/r}, \quad  \text{for  $\beta\neq \ga$,}
\\  \label{eq:Bsprabove}&\|F_{\ga,N} \|_{\dot B^{s}_{p,r} } =\|F_{\ga,N} \|_{ B^{s}_{p,r} } \lc N^{1/r}
\end{align}

(ii) If $p<\infty$ then $F_\ga= \sum_{\ell \ge 1} \ell 2^{-\ell/p}  F_{\ga,2^\ell}$ belongs to
$\bigcap_{r>p} \dot B^s_{p,r}$ and to $\bigcap_{\substack{r>p\\ \beta\neq \ga}} \dot \cB^s_p(\beta,r)$, but not to $\dot \cB^s_p(\gamma,\infty)$.

Also $F_\ga $ belongs to
$\bigcap_{r>p}    B^s_{p,r}$ and to $\bigcap_{\substack{r>p\\ \beta\neq \ga} }   \cB^s_p(\beta,r)$,  but not to $\cB^s_p(\gamma,\infty)$.
\end{proposition}

\begin{proof} Let $r \ge p$. We begin with the upper bound for the $\dot \cB^s_p (\beta,r)$ quasi-norm of $F_{\ga,N}$ for $\beta \ne \ga$. 
Let \[
E_{\gamma, \beta,k} (\la) =\Big\{x\in \bbR^d: \Big|\sum_{i=1}^{\fN_\ga(k)}\eta_{i,k}(x) \Big|^p >\la^p 2^{k(\gamma-\beta) } \Big\}.
\]

Note that from  \eqref{eq:LP-repr} we get 
\[ \|F_{\ga,N}\|_{\dot \cB^s_p(\beta,r)}= 
\Big(r\int_0^\infty \Big[\sum_{k=N+1}^{2N} \la^p 2^{-k\beta} \meas \,E_{\ga,\beta,k}(\la)\Big] ^{r/p} \frac{\dif \la}{\la}
\Big)^{1/p}.
\]
In what follows we will use, for $M>d+|\ga|$, the estimate
\Be\label{eq:etabound} |\eta(x)|\le C_M (1+|x|)^{-M}.\Ee
Split $\eta_{i,k}= \vth_{i,k} + \eps_{i,k}$ where
\[ \vth_{i,k}= \eta_{i,k} \bbone_{\{|x-n_{i,k}e_1| \le 2^k\}},\qquad \eps_{i,k} =\eta_{i,k}-\vth_{i,k}.\]
Note for later reference $\|\eps_{i,k}\|_p \lc_M 2^{-k\frac dp} 2^{2k(\frac dp-M)}$ and therefore 
\Be \label{eq:sumepsik}\Big\| \sum_{i=1}^{\fN_\ga(k) } \eps_{i,k} \Big\|_p \lc 
2^{k(d+\ga)} 2^{-k\frac dp}  2^{2k(\frac dp-M)} .
\Ee
Let 
\begin{align*}
    E^{(1)}_{\gamma,\beta,k}(\la)&= \Big\{x\in \bbR^d: \Big|\sum_{i=1}^{\fN_\ga(k)}\vth_{i,k}(x)\Big|^p>
    (\frac \la 2)^p 2^{k(\ga-\beta)} \Big\}
    \\
    E^{(2)}_{\gamma,\beta,k}(\la)&= \Big\{x\in \bbR^d: \Big|\sum_{i=1}^{\fN_\ga(k)}\eps_{i,k}(x)\Big|^p>(\frac  \la 2)^p 2^{k(\ga-\beta)} \Big\}.
    \end{align*}
Then 
\begin{align}\label{eq:la-decr-inclusion}
    E_{\gamma,\beta,k}(\la) &\subset E^{(1)}_{\gamma,\beta,k}(\la) \cup E^{(2)}_{\gamma,\beta,k}(\la).
\end{align}
Finally set, for $i=1,\dotsc,\fN_\ga(k)$, 
\begin{align*} E^{(1,i)}_{\gamma,\beta,k}(\la)  &= \{x\in \bbR^d: |\vartheta_{i,k}(x)|^p > ( \frac \la 2)^p 2^{k(\ga-\beta)} \}.
\end{align*} Observe that  for fixed $k$ the sets $\supp(\vth_{i,k})$ are disjoint and therefore the sets
$E^{(1)}_{\gamma,\beta,k}(\la) $ are the disjoint union of the sets
$E^{(1,i)}_{\ga,\beta,k}(\la) $, $i=1,\dotsc,\fN_\ga(k)$.
Now from \eqref{eq:etabound} we get
\[
E^{(1,i)}_{\ga,\beta,k}(\la) \subset \Big\{x: |x-n_{i,k}e_1|\le 2^{-k} \Big( \frac{2^pC_M^p}{\la^p2^{k(\ga-\beta)}} \Big) ^{\frac{1}{Mp} } \Big\} \] and therefore we get for the Lebesgue measure 
\begin{align*} 
\meas\, E^{(1)}_{\ga,\beta,k}(\la)  \le c_d\fN_\ga(k) 2^{-kd} 
\Big( \frac{2^pC_M^p}{\la^p2^{k(\ga-\beta)}} \Big) ^{\frac{d}{Mp} } .
\end{align*}
Hence, using the definition of $\fN_\ga(k)$
\[\la^p 2^{-k\beta} \meas\,E^{(1)}_{\ga,\beta,k}(\la) \le c_d (2C_M)^{\frac{d}{M}} 
(\la^p 2^{k(\ga-\beta)})^{1-\frac d{Mp}}
\]
Using \eqref{eq:etabound} we also see that for $k=N+1,\dotsc, 2N$
\[
E^{(1)}_{\gamma,\beta,k}(\la)  =\emptyset \text{ for } 
(\la/2)^p 2^{k(\ga-\beta)}>2C_M^p.
\]

Hence we get, for $\ga>\beta$ and $r < \infty$,
\begin{align*}
    &\Big(\int_0^\infty\Big[\sum_{\substack {N+1\le k\le2 N\\ \la^p 2^{k(\gamma-\beta)}\le 2C_M^p}} \la^p 2^{-k\beta}\meas\, E^{(1)}_{\ga,\beta,k}(\la)\Big]^{r/p} \frac{\dif \la}{\la}\Big)^{1/r}\le I+II
    \end{align*}
    where
    \begin{align*}
        I&=\Big(\int_0^{C_M(2^{-2N(\ga-\beta)})^{1/p}}
    \Big[\sum_{N+1\le k\le2 N}  \widetilde C_M (\la^p  2^{k(\ga-\beta)})^{1-\frac{d}{Mp}} \Big]^{r/p} \frac{\dif\la}{\la} \Big)^{1/r}
        \\
        II&=\Big( \int_{C_M(2^{-2N(\ga-\beta)})^{1/p}}^{2C_M(2^{-N(\ga-\beta)})^{1/p}} \Big[\sum_{\substack {N+1\le k\le2 N\\ \la^p 2^{k(\gamma-\beta)}\le 2^pC_M^p}}  \widetilde C_M (\la^p  2^{k(\ga-\beta)})^{1-\frac{d}{Mp}} \Big]^{r/p} \frac{\dif\la}{\la} \Big)^{1/r}.
    \end{align*}
  We estimate
  \begin{align*}
      I&\lc \Big(\int_0^{C_M(2^{-2N(\ga-\beta)})^{1/p}}
     (\la^p  2^{2N(\ga-\beta)})^{(1-\frac{d}{Mp})\frac rp}  \frac{\dif\la}{\la} \Big)^{1/r}\lc 1 
     \\
     II &\lc \Big( \int_{C_M(2^{-2N(\ga-\beta)})^{1/p}}^{2C_M(2^{-N(\ga-\beta)})^{1/p}}  \frac{\dif\la}{\la} \Big)^{1/r} \lc N^{1/r}
  \end{align*}
  and it follows that $\|F_{N,\ga}\|_{\dot \cB^s_p(\beta,r)}\lc N^{1/r}$ provided that $\beta<\gamma$.
  
  The calculation for $\gamma<\beta$ is very similar, except  the integration is over  $\la\in [0,C_M (2^{1+2N(\beta-\gamma)})^{1/p} ]$ and the corresponding integrals for the parts $I$ and $II$ are extended from $0$ to $C_M (2^{N(\beta-\gamma)})^{1/p}$ and from 
  $C_M (2^{N(\beta-\gamma)})^{1/p}$ to $C_M (2^{1+2N(\beta-\gamma)})^{1/p}$, respectively.  Again the first term gives an $O(1)$ contribution and the second one an $O(N^{1/r})$ contribution. 
  Summarizing we get
  \begin{align}
    &\Big(\int_0^\infty\Big[\sum_{\substack {N< k\le 2N}} \la^p 2^{-k\beta} \meas\, E^{(1)}_{\ga,\beta,k}(\la)\Big]^{r/p} \frac{\dif \la}{\la}\Big)^{1/r}\lc  N^{1/r}.
    \end{align}
A similar (and easier) calculation shows that one has the corresponding bound when $r = \infty$ as long as $\beta \ne \ga$.
    We now estimate the error term; we show in fact the stronger inequality
    \begin{align}\label{eq:Egabetaerror}
    &\Big(\int_0^\infty \Big[\sum_{\substack {N < k\le 2N}} \la^p 2^{-k\beta}\meas\, E^{(2)}_{\ga,\beta,k}(\la)\Big]^{r/p} \frac{\dif \la}{\la}\Big)^{1/r}\lc 2^{-N} 
    \end{align}
    for $r\ge p$. We discretize the integral in $\lambda$, use the embedding $\ell^p\hookrightarrow \ell^r$,  then the change of variables  $\sigma=\la 2^{k(\ga-\beta)/p}$ and then the formula for the $L^p$-norm via the distribution function  to estimate the left hand side of \eqref{eq:Egabetaerror} by 
    \begin{align*}
    &\lc \Big(\sum_{N < k\le 2N} \int_0^\infty \la^p 2^{-k\beta} \meas\, E^{(2)}_{\ga,\beta,k}(\la) \frac{\dif \la}{\la}\Big)^{1/p}
    \\
            &\lc \Big( \sum_{N < k\le 2N} \int_0^\infty 2^{-k\gamma}\sigma^p \meas \big\{ x: |\sum_{i=1}^{\fN_\gamma(k)} \eps_{i,k}(x)|> \sigma\big\} \frac{\dif \sigma}{\sigma}\Big)^{1/p}
        \\&\lc\Big(\sum_{N < k\le 2N} 2^{-k\ga} \Big\| \sum_{i=1}^{\fN_\ga(k)} \eps_{i,k} \Big\|_p^p\Big)^{1/p}    
        \lc 2^{-N},
    \end{align*}
    here we used \eqref{eq:sumepsik} with $M$ large.
    This finishes the proof of $\|F_{\ga,N} \|_{\dot \cB^s_p(\ga,r) }\lc N^{1/r}$ and since the Fourier transform of $F_{\ga,N}$ is supported where $|\xi|\gg1 $ we may replace $\dot \cB^s_p(\beta,r)$ with $\cB^s_p(\ga,r)$.
    Thus
    \eqref{eq:Bspbetarabove} is now proved, and  this inequality also yields
   $ \|F_\gamma\|_{\dot \cB^s_p (\beta,r) }\lc 1$.

    We now give the proof of \eqref{eq:Bsprabove}.  The proof is similar to the above but simpler.
    We use \eqref{eq:LP-repr} to write
    \[\|F_{\ga,N}\|_{\dot B^s_{p,r}} =\Big(\sum_{k=N+1}^{2N} \|f_{\ga,k} \|_p^r\Big)^{1/r} =I_1+II_1\]
    where
    \begin{align*} 
    I_1&= \Big(\sum_{k=N+1}^{2N} \Big\|2^{-k\ga/p} \sum_{i=1}^{\fN_\ga (k)} \vth_{i, k} \Big \|_p^r\Big)^{1/r}
    \\
    II_1&\lc 
     \Big(\sum_{k=N+1}^{2N} \Big \|2^{-k\ga/p} \sum_{i=1}^{\fN_\ga (k)} \eps_{i, k} \Big\|_p^r\Big)^{1/r}
    \end{align*}
    Using the disjointness of support property of the $\vth_{i,k}$ we have
    \[ \|2^{-k\ga/p} \sum_{i=1}^{\fN_\ga (k)} \vth_{i, k} \|_p\approx 1, \quad N< k\le 2N\]
    and hence $I_1\approx N^{1/r}$ (with the obvious modification if $r = \infty$).
    For $II_1$ we use \eqref{eq:sumepsik} for sufficiently large $M$ and see that $|II_1|\lc 2^{-N}$, and \eqref{eq:Bsprabove} follows. We also have $\|F_\ga\|_{\dot B^s_{p,r}} \lesssim 1$.
    
    We conclude by proving the lower bound 
    \eqref{eq:Bspgarbelow}.
    We have for $\la\ll 1$
    \begin{align*} 
    \|F_{\ga,N} \|^p_{\dot \cB^s_p(\ga,\infty)} &\ge \sum_{N < k \le 2N}\la^p2^{-k\gamma} \meas\Big\{x: \big|\sum_{i=1}^{\fN_\ga(k)} \eta_{i,k}(x)\big|>\la\Big\} \ge I_2^p-II_2^p \end{align*}
    where
    \begin{align*}
    I_2^p&=\la^p\sum_{N < k \le 2N}2^{-k\gamma} \meas\Big\{x: \big|\sum_{i=1}^{\fN_\ga(k)} \vth_{i,k}(x)\big|>2\la\Big\}
    \\
    II_2^p&=\la^p\sum_{N < k \le 2N}2^{-k\gamma} \meas \Big\{x: \big|\sum_{i=1}^{\fN_\ga(k)} \eps_{i,k}(x)\big|>\la\Big\}.
    \end{align*}
By the support properties of the $\vth_{i,k}$ and by \eqref{eq:etaapprox} we have for sufficiently small $\la$ 
\[\meas \Big\{x: \big|\sum_{i=1}^{\fN_\ga(k)} \vth_{i,k}(x)\big|>2\la\Big\} \gc \fN_\ga(k) 2^{-kd}\approx 2^{k\ga} \] and hence $I_2\gc  N^{1/p}$.
 By \eqref{eq:sumepsik} and Chebyshev's inequality 
 \begin{align*}
     II_2^p\lc 
     \sum_{N < k \le 2N} 2^{-k\ga} \Big\|\sum_{i=1}^{\fN_\ga(k)} \eps_{i,k} \Big\|_p^p
     \lc 2^{-N} 
 \end{align*} and combining the two estimates we get for sufficiently large $N$ the desired lower bound $\|F_{\ga,N} \|_{\dot\cB^s_p(\ga,\infty)} \ge cN^{1/p}$ .
 
 Finally 
 \begin{align*}
 \|F_\ga\|_{\dot\cB^s_p(\ga,\infty)} &\gc \sup_{\la>0} \la\Big(
 \sum_{\ell \ge 1}  \sum_{2^{\ell} < k \le 2^{\ell+1}} 2^{-k\gamma}
 \meas\Big\{  \ell 2^{-\ell/p} \big|\sum_{i=1}^{\fN_\ga(k)} \eta_{i, k}\big |>\la \Big\} \Big)^{1/p}
 \\
 &\ge \sup_{\ell \ge 1} \ell 2^{-\ell/p} \sup_{\sigma>0} \sigma\Big(
  \sum_{2^{\ell} < k \le 2^{\ell+1}} 2^{-k\gamma}
 \meas\Big\{ \big|\sum_{i=1}^{\fN_\ga(k)} \eta_{i, k}\big |>\sigma \Big\} \Big)^{1/p}
 \\
 &\gtrsim \sup_{\ell \ge 1} \ell 2^{-\ell/p} 
 \sigma_0 (2^{\ell})^{1/p}  =\infty
 \end{align*}
for sufficiently small $\sigma_0\ll 1$,  and we see that $F_\ga\notin \dot\cB^s_p(\ga,\infty) $.
\end{proof}

The counterpart of Proposition \ref{prop:counterex-ga>-d} in the range  $\gamma<-d$ is as follows.
\begin{proposition} \label{prop:counterex-ga<-d}
 Let $f_{\ga,k}$ be as in \eqref{eq:fgak}. Let $s\in \bbR$.
 Assume $\gamma<-d$ and define
 \begin{subequations}\label{eq:FG-defs}
\Be \label{eq:FgaN-kneg} F_{ \ga,N} (x) = 
\sum_{-2N < k \le -N} 2^{-ks}f_{\ga,k}(x) 
\Ee  and
\Be\label{eq:GgaN} 
 G_{\ga,N}(x)= 2^{3N(\frac dp-s)} F_{\ga,N}(2^{3N} x). 
\Ee
\end{subequations}

(i) Then for $1 < p \leq r \leq \infty$,
\begin{align*}
&\|G_{\ga,N} \|_{ \cB^s_p (\gamma,\infty)} =\|F_{\ga,N} \|_{\dot \cB^s_p (\gamma,\infty)} 
\gc N^{1/p},
\\&\|G_{\ga,N} \|_{ \cB^s_p (\beta,r)} =\|F_{\ga,N} \|_{\dot \cB^s_p (\beta,r)}          \approx N^{1/r}  , \quad \text{for } \beta\neq \gamma,
\\  &\|G_{\ga,N} \|_{ B^{s}_{p,r} } =\|F_{\ga,N}\|_{\dot B^{s}_{p,r} }  \approx N^{1/r}.
\end{align*}

(ii) If $p < \infty$, then $F= \sum_{\ell\ge 1} \ell 2^{-\ell/p}   F_{\ga,2^\ell}$ belongs to
$\bigcap_{r>p} \dot B^s_{p,r}$ and to $\bigcap_{\substack{r>p\\ \beta\neq \ga}} \dot \cB^s_p(\beta,r)$ but not to $\dot \cB^s_p(\gamma,\infty)$.

(iii) If $p < \infty$, then
$G_\ga= \sum_{\ell\ge 1} \ell 2^{-\ell/p} G_{\ga, 2^\ell}$
belongs to
$\bigcap_{r>p}  B^s_{p,r}$ and to $\bigcap_{\substack{r>p\\ \beta\neq \ga}}  \cB^s_p(\beta,r)$ but not to $ \cB^s_p(\gamma,\infty)$.
\end{proposition}

\begin{proof}[Sketch of proof]
The proof of the bounds for $F_{\ga,N}$ is exactly analogous to the corresponding arguments in Proposition \ref{prop:counterex-ga>-d}.  Observe that the parameter $k$ now varies between $-2N$ and $-N$ and since $\gamma<-d$ we now have $\fN_\ga(k)= \lfloor 2^{k(d+\gamma)}\rfloor= \lfloor 2^{|k||d+\gamma|}\rfloor \ge 1$.
Also notice that the Fourier transform of $G_{\ga,N}$ is supported on large frequencies and therefore the homogeneous and inhomogeneous  Besov type norms for $G_{\ga,N}$ coincide. 

To pass from estimates for $F_{\ga,N}$ to estimates for $G_{\ga,N}$ we just use the dilation formulas
\begin{align*} &2^{n(\frac dp-s)}  \|f(2^n\cdot) \|_{\dot\cB^s_p (\ga,r)} = 
\|f\|_{\dot\cB^s_p (\ga,r)},
\\
&2^{n(\frac dp-s)}  \|f(2^n\cdot) \|_{\dot B^s_{p,r}} = 
\|f\|_{\dot B^s_{p,r}}.
\qedhere
\end{align*}
\end{proof}

The following  two lemmas show 
that the assumption  $\gamma\neq 0$ in part (ii) of Theorem
\ref{thm:embeddings} cannot be removed. A combination of these lemmas gives  a proof of  Theorem \ref{thm:non-emb-gamma}.

\begin{lemma}\label{prop:F-B} Let $s\in \bbR$, and $1<p<\infty$. 
There exists $f\in \bigcap_{r>p} \dot F^s_{p,r}$ which does not belong to $\dot \cB^s_p(0,\infty)$.
\end{lemma}

\begin{proof}
Let $\eta_\circ$ be a Schwartz function such that $\eta_\circ(x)>1$ for $|x|<1$ and $\widehat \eta_\circ $ is supported in $\{\xi:|\xi|\le 2^{-5} \}$.
For $k>2$ let
\[ f_k(x)=\eta_\circ(x) e^{i2^k x_1} \frac{\log k}{k^{1/p}} \]
and $f(x)=\sum_{k>2} 2^{-ks} f_k(x)$. 
Then $L_k f=2^{-ks} f_k $  and thus 
\begin{align*}
    &\Big(\sum_{k>2} 2^{ksr} |L_k f(x)|^r\Big)^{1/r} = \Big(\sum_{k>2} |f_k(x)|^r\Big)^{1/r}
    \\
    &\lc |\eta_\circ(x)|  \big(\sum_{k>2} k^{-r/p} |\log k|^r\big)^{1/r} \lc C(p,r) |\eta_\circ(x)|
\end{align*}
with $C(p,r)<\infty$ for $r>p$. Hence $f\in \dot F^s_{p,r}$ for all $r>p$.

For $\la\ll1$ we have 
\begin{align*}
   & \la \mu_0\{(x,k): |P^sf(x,k)|>\la\}^{1/p}
    \\
    &\ge\la  \Big(\sum_{\substack{k>2\\ k^{-1/p}\log k>\la}} \meas\{x: |x|<1/4\} \Big)^{1/p} 
    \\
    &\ge c\la \Big( \sum_{2<k<c (\frac {\log \la^{-1}} \la)^p  }1 \Big)^{1/p} \ge c \log \la^{-1}
\end{align*}
so that $f$ does not belong to $\dot \cB^s_p(0,\infty)$.
\end{proof}




\begin{lemma}\label{prop:F-B-r=1} Let $s\in \bbR$, and $1<p<\infty$. 
There exists  $g\in \dot\cB^s_p(0,1)$ which does not belong to $ \bigcup_{r<p} \dot F^s_{p,r}$.
\end{lemma}

\begin{proof} 
 As in the proof of Lemma \ref{prop:F-B} let $\eta_\circ$ be a Schwartz function such that $\eta_\circ(x)>1$ for $|x|<1$ and $\widehat \eta_\circ $ is supported in $\{\xi:|\xi|\le 2^{-5} \}$.
For $k>2$ let
\[ g_k(x)=\frac{\eta_\circ(x) e^{i2^k x_1} } {k^{1/p} [\log k] ^2} \]
and $g(x)=\sum_{k>2} 2^{-ks} g_k(x)$. 
Then $L_k g=2^{-ks} g_k $  and thus 
\begin{align*}
    &\Big(\sum_{2<k\le 2N}  2^{ksr} |L_k g(x)|^r\Big)^{1/r} = \Big(\sum_{2<k\le 2N}  |g_k(x)|^r\Big)^{1/r}
    \\
    &\ge |\eta_\circ(x)|  \big(\sum_{N\le k\le 2N} k^{-r/p} |\log k|^{-2r}\big)^{1/r} \ge  C(p,r) N^{1-r/p} (\log N)^{-2} |\eta_\circ(x)|
\end{align*}
with $C(p,r) >0$ and $1-\frac{r}{p} > 0$ for $r<p$. Integrating its $p$-th power over $\{x:|x|\le 1/2\}$ and letting $N\to\infty$ we see that  $f\notin \dot F^s_{p,r}$ for all $r<p$.

Now let,  
\begin{align*} E_{\ell,m} &=\{(x,k): 2^{\ell-1}\le |x|< 2^{\ell},\,\, 2^{m-1}\le k < 2^m\}, \text{ for $(\ell,m)\in \bbN^2$}, 
\\
E_{0,m}&=\{(x,k): |x|<1, \,\, 2^{m-1}\le k < 2^m\}, \text{ for  $m\in \bbN$}. \end{align*}
Then $\mu_0(E_{\ell,m})\approx  2^{m+\ell d}$ and 
\[|g_k(x) |\lc_N
2^{-m/p} m^{-2} 2^{-\ell N}  \text{ if }  (x,k)\in E_{\ell,m}, 
\] for any $(\ell,m)\in (\bbN\cup\{0\})\times \bbN$.  Therefore 
\[|P^s g(x,k) |=|g_k(x)| \lc_N \sum_{\ell=0}^\infty \sum_{m=1}^\infty 2^{-\ell (N-d/p)}  m^{-2} \frac{\bbone_{E_{\ell,m}}(x,k) } {\mu_0 (E_{\ell,m})^{1/p}}.\]
Choosing  $N>d/p$ we see that  $P^s g\in L^{p,1}(\mu_0)$ and since $\widehat g(\xi)=0$ for $|\xi|\le 1$ we obtain  $g\in \dot \cB^s_p(0,1) $.
\end{proof}
\section{Proof of Theorem  \ref{thm:BVemb} } \label{sec:BV-interpol}
We use a result in \cite{bsvy}, namely for \(\gamma \in \R \setminus [-1,0]\) 
\Be\label{eq:bsvy}
[ \cQ_{1,1+\gamma} f ]_{L^{1, \infty} (\nu_\gamma)} \lesssim \|f\|_{\dot {BV}(\R^d)}.
\Ee
Since \(\abs{\cQ_{1,(1+\ga)/p} f}^p \le \abs{\cQ_{1,1+\ga} f}\,  ( 2\Norm{f}_{V^\infty})^{p - 1},\) we have \[\|f\|_{\dot \cB^{1/p}_p(\gamma,\infty)} \simeq [\cQ_{1,\frac{1+\ga}{p}} f]_{L^{p,\infty}(\nu_{\ga})} \lesssim \|f\|_{V^{\infty}}^{1-\frac{1}{p}} [\cQ_{1, 1+\ga} f]_{L^{1,\infty}(\nu_{\ga})}^{\frac{1}{p}},\] which combined with \eqref{eq:bsvy}  gives
\Be \label{eq:BVembedding}\|f\|_{\dot \cB^{1/p}_p(\gamma,\infty)} \lc  \|f\|_{V^{\infty}}^{1-\frac 1p} \|f\|_{\dot {BV}}^{\frac 1p }\Ee
for every \(\gamma \in \R \setminus [-1,0]\) and \(1 < p <\infty\).

We can interpret   inequality \eqref{eq:BVembedding} as an imbedding result for the real interpolation space  $[V^\infty, \dot{BV}]_{\theta,1}$ with $\theta=\frac 1p$, and get,  
\Be \label{eq:BVembedding-alt}\|f\|_{\dot \cB^{1/p}_p(\gamma,\infty)} \lc \|f\|_{[V^\infty, \dot {BV} ]_{\frac 1p,1 }}.\Ee
For this argument, see \cite[p. 49]{BL} (specifically a combination of formula (1) and Theorem 3.5.2 (b)).
 \qed
 


\begin{remarka} Concerning  Remark \ref{rem:Vpinterpol}, by the reiteration theorem (\cite[Theorem 3.5.3]{BL}) we have 
$ \big [  [V^\infty,\dot{BV}]_{\frac 1{p_0},1},
 [V^\infty,\dot{BV}]_{\frac 1{p_1},1}
 \big]_{\theta,\infty} = [V^\infty, \dot{BV}]_{\frac 1p,\infty} $, provided that  $1<p_0<p<p_1<\infty$ and  $\frac 1p=\frac{1-\theta}{p_0} +\frac\theta{p_1}$.
 Hence $[V^\infty, \dot{BV}]_{\frac 1p,\infty}$ embeds only into $[\dot \cB_{p_0}^{1/p_0} (\ga,\infty) , \dot \cB_{p_1}^{1/p_1}(\ga,\infty)]_{\theta,\infty} $ 
 (a weaker conclusion than embedding into $\dot \cB^{1/p}_p(\ga,\infty)$).
\end{remarka} 
 
\section{Harmonic and caloric extensions}
\label{sec:harmonic-caloric}
In what follows let $\psi$ be a sufficiently well behaved integrable function with $\int \psi(x)\dif x=0$, specifically we will take $\psi$ as one of $\psi_{(1)}$, $\psi_{(2,j)}$, $\psi_{(3)}$, $\psi_{(4,j)}$ 
where
\Be\label{eq:choices-of-psi} 
\begin{aligned} &\widehat \psi_{(1)}(\xi)=|\xi|e^{-|\xi| }, \qquad \widehat \psi_{(2,j)} (\xi)= i\xi_j e^{-|\xi|}, 
\\
&\widehat \psi_{(3)}(\xi)= -|\xi|^2 e^{-|\xi|^2}, 
\qquad 
\widehat \psi_{(4,j)} (\xi) = i\xi_j e^{-|\xi|^2},
\end{aligned}\Ee or we could also take  $\psi=\psi_{(5)}=  \frac{\partial}{\partial s} [s^{-d} \phi(s^{-1} \cdot)] |_{s=1}$ for any  $\phi\in \cS(\bbR^d)$. 
Let  $\psi_t=t^{-d}\psi(t^{-1}\cdot) $.
Classical results on characterizations of Besov spaces
(\cite{taibleson}, \cite[Chapter V.5, Proposition 7']{stein-diff}, \cite[Chapter 1.8]{triebel2})  yield the inequality
\Be\label{eq:classical} 
\iint_{\R^{2}_+} t^{-sp} \abs{\psi_t*f(x)}^{p} \dif x \frac{\dif t}{t} \lc \|f\|_{\dot W^{s,p}}^p
\Ee for $f \in \dot W^{s,p}$, $1\le p<\infty$ and $0<s<1$.

With  $\psi$, $\psi_t$   as above,   define, for $f \in \dot \cB^s_p(\ga,r)$ with $0 < s < 1$,  \[\cK^b f(x,t) = t^{-b} \psi_t*f_{\circ}(x)\]
where $f_{\circ}$ is any representative of $f$ modulo constants. 
Recall from \eqref{eq:lagameas} that $\dif\la_\ga(x,t) = t^{\ga-1} \dif t\dif x$.

\begin{proposition}
\label{proposition_space-time_extension}
Let \(0 < s < 1\), \(1 < p, r < \infty\) and \(\gamma \in \R\).
The operator \(\cK^{s + \frac{\gamma}{p}}\) is bounded from \(\dot\cB^s_p(\ga,r)\) to \(L^{p, r} (\lambda_\gamma)\).
\end{proposition}
\begin{proof}
 We take \((s_0, p_0)\), \((s_1, p_1)\) and \(\theta\) such that \eqref{eq_gamma} and \eqref{eq_theta} holds, and \(0 < s_0 < 1\), \(0 < s_1 < 1\), \(1 < p_0 < \infty\) and \(1 < p_1 < \infty\).
 Recall $s+\frac{\ga}{p} = s_i + \frac{\ga}{p_i}$, and observe that for $f \in \dot B^{s_i}_{p_i,p_i}$, 
 \begin{equation*}
 \iint_{\R^{2}_+} \frac{\abs{\cK^{s+\frac{\gamma}{p}}f (x, t)}^{p_i}}{t^{1 - \gamma}} \dif x \dif t
 = \iint_{\R^{2}_+} t^{-s_i p_i} \abs{\psi_t*f_{\circ} (x)}^{p_i} \dif x \frac{\dif t}{t}.
\end{equation*}
where $f_{\circ} \in \dot W^{s_i,p_i}$ is a representative of $f$ modulo constants.
 It follows from \eqref{eq:classical} that  \(\cK^{s + \frac{\gamma}{p}}\) 
is bounded from \(\dot B^{s_i}_{p_i, p_i}\) to \(L^{p_i} (\lambda_\gamma)\).  
The conclusion then follows by interpolation, in view of Theorem~\ref{interpol-thm-Fourier} and of the  classical characterization of Lorentz spaces as interpolation spaces.
\end{proof}

\begin{corollary} \label{cor:CK-BV} Let $1<p<\infty$, $\gamma\in \bbR\setminus [-1,0] $. Then
\[\cK^{\frac{\ga+1}{p}} : \, [ V^\infty, \dot {BV}]_{\frac 1p,1}\to L^{p,\infty}(\la_\ga)\]
is bounded.
\end{corollary}
\begin{proof}
Combine Proposition \ref{proposition_space-time_extension} for $s=1/p$ with Theorem \ref{thm:BVemb}.\end{proof}


\subsection{\it Harmonic extensions: Proof of Corollary \ref{cor:Poisson-emb}} From \cite[Lemma 1.17]{stein-weiss} we recall that $\widehat {\poi f}(\xi,t)=e^{-t|\xi|}\widehat f(\xi)$  and therefore we are led to use  
the function $\psi_{(1)} $ and $\psi_{(2,\nu)}$ for $\nu=1,\dots, d$ in \eqref{eq:choices-of-psi}, for formulas for $t\frac{\partial}{\partial t} \poi f$ and $t\frac{\partial}{\partial x_\nu} \poi f $, respectively. 
We let
$\mathscr K^b f(x,t)= t^{1-b}\nabla \mathscr P f(x,t)$ and  apply Corollary 
\ref{cor:CK-BV} to  obtain  
\[ \|\mathscr K^{\frac{\ga+1} p} f\|_{L^{p,\infty}(\la_\ga)} \lc \|f\|_{[ V^\infty, \dot {BV}]_{\frac 1p,1}}\]
and the proof of the first inequality in Corollary \ref{cor:Poisson-emb} is  complete. For the proof of the  second inequality choose $\gamma=1$ and $p=2$,  which is the unique choice of $p,\ga$ where $\mathscr K^{\frac{\ga+1}{p} }$  becomes $\nabla \mathscr P $ and $\la_\ga$ becomes Lebesgue measure on $\bbR^{d+1}_+$. \qed

\subsection{\it Caloric  extensions:  Proof  of Corollary \ref{cor:heat} }
Note  that \(r\frac{\partial}{\partial r}[ \widehat {Uf}(\xi,r^2)]\) equals \(  2 |r\xi|^2e^{-| r\xi|^2 } \widehat f(\xi),\)  and taking  $\psi=\psi_{(3)}$ in the definition of $\cK^b\equiv \cK_{d+1}^b$,  we get 
\[2\cK^b_{d+1} f(x,r)= r^{1-b} \frac{\partial}{\partial r}[Uf(x,r^2)]
=2 t^{1-\frac b2 }  \frac{\partial }{\partial t} Uf (x,t) \Big|_{t=r^2}
=2\mathscr H_{d+1}^{b/2} f (x,r^2).
\]  
We apply  Corollary \ref{cor:CK-BV}  with  $\ga=2\beta$ and observe that
 \[\la_{2\beta}(\{(x,r): |\cK_{d+1}^{\frac{2\beta+1}{p}} f(x,r)|>\alpha\})= \frac 12 \la_\beta(\{(x,t): \mathscr H_{d+1}^{\frac{2\beta+1}{2p}} f(x,t)>\alpha\}).
\] For $2\beta\notin [-1,0]$ the operator   $\cK^{\frac{2\beta+1}{p}} $ maps 
$[V^\infty,\dot{BV}]_{1/p,1} $ to $L^{p,\infty}(\la_{2\beta})$, and hence  
$\mathscr H_{d+1}^{\frac{2\beta+1}{2p} }$ maps $[V^\infty,\dot{BV}]_{1/p,1} $ to 
$L^{p,\infty}(\la_{\beta})$.

For $j=1,\dots, d$ we argue similar, taking $\psi=\psi_{(4,j)}$ in the definition of $\cK^b\equiv \cK^b_j$. We then have 
\(\cK^b_j f(x,r)= r^{1-b} \frac{\partial }{\partial x_j} Uf(x,r^2)\) and again apply Corollary \ref{cor:CK-BV} with $\ga=2\beta$. Now for $j=1,\dots, d$, 
\[
\cK_j^bf(x,r)=
t^{\frac{1-b}{2} }  \frac{\partial }{\partial x_j} Uf (x,t)\Big|_{t=r^2}
=\mathscr H_j^{b/2} f (x,r^2), \] and again we see  that
$\mathscr H_{j}^{\frac{2\beta+1}{2p} }$ maps $[V^\infty,\dot{BV}]_{1/p,1} $ to 
$L^{p,\infty}(\la_{\beta})$.
This finishes the proof of  part (i) of the corollary. For part (ii) we need $\dif \lambda_{\beta} = \dif x \dif t$ so that we put  $\beta=1$. We then apply  part (i), for the operator $\frac{\partial U}{\partial t }$ with  $p=3/2$ (so that  $\frac{2\beta+1}{2p} = 1$),  and for the operators $\frac{\partial U}{\partial x_j} $
with  $p=3$ (so that  $\frac{2\beta+1}{2p} = \frac 12$). \qed

\section{Interpolation: Proof of 
Theorem \ref{interpol-thm-Fourier}}\label{sec:interpolation}
We use the standard  retraction-coretraction argument (see \cite[\S6.4]{BL}).
Recall that if $\overline X=(X_0,X_1)$ and $\overline Y=(Y_0,Y_1)$ are  couples of compatible normed spaces  then 
 $P:\overline X\to \overline Y$  is a morphism of couples if $P$ is a linear operator mapping  $X_0+X_1$ to $Y_0+Y_1$, such that  $P:X_\nu\to Y_\nu$  is a bounded linear operator for $\nu=0$ and $\nu=1$. 
 
 If $ P: \overline X\to \overline Y$, $R:\overline Y\to \overline X$ are be morphisms of couples such that  $R\circ P: \overline X\to \overline X=Id$,  the identity operator on $\overline X$ then $\overline X$ is called a retract of $\overline Y$;  $R$  is a retraction and $P$ is a co-retraction.
  \begin{center}
\begin{tikzcd}
& \overline Y\arrow[dr,"R"] & \\ \overline X \arrow[hook,ur,"P"] \arrow[rr,"id"] && \overline X
\end{tikzcd}
\end{center}

\begin{lemma}  \label{lem:retract-formalism} \cite{BL}
Let $\overline X=(X_0,X_1)$, $\overline Y=(Y_0,Y_1)$ be a couples of compatible normed spaces such that $\overline X$ is a retract of $\overline Y$ with co-retraction $P:\overline X\to \overline Y$ and retraction $R$ then
\[ [X_0,X_1]_{\theta,r}= \{ f\in X_0+X_1:  Pf\in [Y_0,Y_1]_{\theta,r} \} \] 
and we have the equivalence of norms, 
$\|f\|_{[X_0, X_1]_{\theta,r} }\approx \|Pf\|_{[Y_0,Y_1]_{\theta,r} }  .$
\end{lemma}


\begin{lemma} \label{lem:retract}
Suppose $1 \le p_0<p_1\le \infty$, 
and $\ga, b \in \R$. 
Then  there are bounded morphisms of couples 
\begin{align*}
    P^b&:  (\dot B^{b-\frac \ga{p_0}} _{p_0,p_0} , \dot B^{b-\frac \ga {p_1}}_{p_1,p_1} ) \to (L^{p_0}(\mu_\gamma) , L^{p_1} (\mu_\gamma))
    \\
    R_b &: (L^{p_0} (\mu_\gamma) , L^{p_1} (\mu_\gamma))
    \to (\dot B^{b-\frac \ga{p_0}} _{p_0,p_0} , \dot B^{b-\frac \ga {p_1}}_{p_1,p_1} ) 
\end{align*}

\end{lemma}

\begin{proof} 
The definitions of $P^b$, $R_b$ will be independent of $p_0, p_1$ and thus one can reduce to  checking  the boundedness of
\begin{align}  \label{eq:iso} 
P^b&: \dot B^{b-\frac \gamma p}_{p,p} \to L^p(\mu_\ga)
\\ \label{eq:Rbbound} 
 R_{b}&:  L^{p}(\mu_\ga)\to \dot B^{b-\frac \ga p} _{p,p} 
\end{align}for $1\le p\le \infty$.


Recall $L_k= \varphi(2^{-k}D)$, $\widetilde  L_k=\widetilde  \varphi(2^{-k}D) $ with $\varphi$ as in \eqref{eq:varphiassumptions} and $ \widetilde \varphi$  as
in \eqref{eq:Tildephi-assu}, satisfying  $\widetilde L_k=\widetilde L_k L_k= L_k\widetilde L_k$. Let $P^b$ be as in 
Definition \eqref{first-definition}. 
For $F\in L^{p,r} (\mu_\ga) $ define $F_k(x) \coloneqq F(x,k)$ and
$R_b F(x)= \sum_{k=0}^\infty 2^{-kb} \widetilde L_k F_k(x).$

Note that $
P^b: \dot B^{b-\gamma/p}_{p,p} \to L^p(\mu_\ga)$ is an 
 isometric embedding, for $1\le p\le \infty$;
moreover  
$ R_b  P^{b} $ is the identity on $\dot\cB^{b-\frac{\ga}{p}}_{ p,p} $.
It remain to show that   
$ R_{b}$ maps $  L^{p}(\mu_\ga)$ boundedly to $\dot B^{b-\frac \ga p} _{p,p} $.
Indeed we  have $L_k\widetilde L_{k+j}=0$ for $|j|>2$ and thus
\[ 
2^{kb} L_k R_bF(x)=2^{kb}\sum_{j=-1}^{1} 2^{-(k+j)b}  L_k\widetilde L_{k+j} F_{k+j}(x)
\]
and, defining 
$T_j F(x,k)= L_k\widetilde L_{k+j} F_{k+j}(x) $
for $j=-1,0,1$, 
we see that
\[\|R_b F\|_{\dot B^{b-\frac \ga p}_{p,p}}
\lc \sum_{j=-1}^{1} \|T_j F\|_{L^{p}(\mu_\ga) }.
\] 
and the boundedness  of $R_b$
follows from 
\begin{align*}&\|T_j F\|_{L^{q}(\mu_\ga) } = \Big(\sum_{k\in\bbZ} 
 \|L_k\widetilde L_{k+j}  F_{k+j} \|_{L^q}^q 2^{-k\ga} \Big)^{1/q} 
 \\&\lc 
 \Big(\sum_{k\in\bbZ} 
 \| F_{k+j} \|_{L^q}^q 2^{-(k+j)\ga} \Big)^{1/q} 
 \lc \|F\|_{L^{q} (\mu_\ga)}, \quad j=-1,0,1. \qedhere\end{align*}
 \end{proof}

\begin{proof} [Proof of Theorem \ref{interpol-thm-Fourier}, conclusion]
Our choice of $\ga$ allows us to define 
\[
b := s_0 + \frac{\ga}{p_0} = s_1 + \frac{\ga}{p_1}.
\]
We apply Lemma \ref {lem:retract-formalism} with $X_\nu=\dot B^{ b-\gamma/p_\nu} _{p_\nu,p_{\nu}}$, $Y_\nu=L^{p_\nu,r}(\mu_\gamma)$, 
$\nu=0,1$ and $P=P^b$, $R=R_b$, as in Lemma \ref{lem:retract}. We then  use the standard interpolation formula  $[L^{p_0}, L^{p_1}]_{\theta,r}=L^{p,r}$  for $(1-\theta)/p_0+\theta/p_1=1/p$, see \cite{BL}, and the definition $\dot B^s_p(\gamma, r)$ via the operator $P_b$.
\end{proof} 
\begin{proof}[Proof of Corollary \ref{cor:TLinterpol}]
\eqref{eq:cBinterpol}  
follows from Theorem \ref{interpol-thm-Fourier}
by the reiteration theorem  for the real method. 
\eqref{eq:TLinterpol}  for general $q_0,q_1$ 
follows since for 
$1\le r_i\le \infty$ and $\ga \ne 0$ given by \eqref{eq_gamma} we have by  part (ii) of Theorem \ref{thm:embeddings}  
\[ \dot \cB^{s_i}_{p_i}(\gamma,1)\hookrightarrow 
 \dot F^{s_i}_{p_i,1} \hookrightarrow
\dot F^{s_i}_{p_i,r_i} \hookrightarrow \dot F^{s_i}_{p_i,\infty}  \hookrightarrow \dot\cB^{s_i}_{p_i}(\gamma,\infty). \qedhere
\]
\end{proof}

\emph{Remark.} 
Asekritova and Kruglyak \cite{ak} obtained real interpolation results for triples of 
the Besov spaces $(B^{s_0}_{p_0,p_0}, B^{s_1}_{p_1,p_1}, B^{s_2}_{p_2,p_2})_{\vec \theta,r}$, with $\sum_{i=0}^2\theta_i=1$ (or more generally $(\ell+1)$-tuples of such spaces with $\ell\ge 2$).  Under the crucial additional assumption that the three points points 
$(\tfrac{1}{p_i},s_i)$, $i=0,1,2$ 
 \emph{do not lie on a line}  the interpolation spaces is identified with the {Besov-Lorentz space $B^s_r(L^{p,r})$}   where $(\frac 1p,s)=\sum_{i=0}^2\theta_i (\frac{1}{p_i},s_i) $. 
 The result for triples does not seem to have an implication on the interpolation of \emph{couples} of Besov spaces (see  also \cite{ak-proc}).

\medskip

\emph{Remark.} One could also use more directly  results on real interpolation of weighted spaces, namely the identification of $[L^{p_0}(w_0), L^{p_1}(w_1)]_{\theta,q} $ in work by 
Freitag \cite{freitag} and by   Lizorkin \cite{lizorkin}.

\section{Interpolating  Besov spaces through differences}
\label{sec:interpol-diff}
In this section  we provide a direct proof of
\eqref{eq:interpol_nice} in the case $M = 1$,
which is directly based on the characterization using first differences. Suppose 
$0<s<1$, $1<p<\infty$, $1\le r\le\infty$, for  $p_0<p<p_1$ so that 
$s_i\coloneqq s+\ga(\frac 1p-\frac 1{p_i})$ satisfy $0<s_i<1$ for  $i=0,1$, and $\theta\in (0,1)$ such that $\frac{1-\theta}{p_0}+\frac{\theta}{p_1}=\frac 1p$. We will prove that for all 
functions \(f : \R^d \to \bbC\) in
$\dot W^{s_0,p_0} +\dot W^{s_1,p_1} $, 
\Be \label{eq:interpolation-diff}
\|\cQ_{1,s+\frac \ga p} f\|_{L^{p,r} (\nu_\ga)} \approx  \|f\|_{[\dot
W^{s_0,p_0} ,
\dot W^{s_1,p_1} ]_{\theta,r} }\,
\Ee 
The alternative proof goes by  a retraction argument based on  differences. One uses Lemma \ref{lem:retract-formalism} once the  following proposition is established. 

\begin{proposition} \label{prop:retract-diff}
Let $b\in \bbR$ with $0<b-\ga/p<1$. 
There is a bounded operator $A_b: L^p(\nu_\gamma)\to\dot W^{b-\frac \ga p,p} $ such that  $A_b\cQ_{1,b} $ is the identity on $\dot W^{b-\frac \ga p,p}.
$ \end{proposition}
That is, we have the following retract diagram
\begin{center}
\begin{tikzcd}
& L^{p}(\nu_\ga) \arrow[dr,"A_b"] & \\ \dot W^{b-\frac \ga p, p}  \arrow[hook,ur,"\cQ_{1,b}"] \arrow[rr,"id"] && \dot W^{b-\frac \ga p, p}
\end{tikzcd}
\end{center}
The proof of the proposition is inspired by the metric characterization of sums $\dot W^{s_0,p_0}+\dot W^{s_1,p_1}$ due to Rodiac and the fourth named author \cite{Rodiac_VanSchaftingen_2021}.

Fix $\ga \in \R$ and $1 \leq p < \infty$. Fix $\phi \in C^{\infty}_c(\R^d)$ with $\int \phi = 1$ and support inside $B_{1/2}(0)$ and let
\[\psi(y) \coloneqq -\phi(y)d-\inn{y}{\nabla\phi(y)}.\]
Integration by parts shows that $\int\psi(y)\dif y=0$. 
For $t > 0$ define $\phi_t(y) \coloneqq \frac{1}{t^d} \phi(\frac{y}{t})$ and 
$\psi_t(y) \coloneqq \frac{1}{t^d} \psi(\frac{y}{t})$; one verifies that for all $t>0$
\begin{equation}\label{eq:Phiprop}
 \psi_t(y) = t \frac{d}{dt} \phi_t(y).
\end{equation}
In what follows we set \Be\label{eq:vthtdef} \vartheta_t(z,y) \coloneqq \phi_t(z) \psi_t(y).\Ee 
Suppose that  $F\in L^p(\nu_\gamma)$
is compactly supported in $\bbR^d\times(\bbR^d \setminus \{0\}) $. We then define 
\begin{multline}\label{eq:Abepsdef}A_{b,\eps} F(x)= \int_{\varepsilon}^{1/\varepsilon} \int_{\R^d} \int_{\R^d}  F(y,h) |h|^b \vartheta_t(x-y-h, x-y)  \dif h \dif y  \frac{\dif t}{t}.
\end{multline} 
Since $\vartheta$ is supported in $B_{1/2}(0)\times B_{1/2}(0)$ it is clear that for $\eps>0$  and
for $F$ with the above support property
the integral in \eqref{eq:Abepsdef} converges absolutely, and defines  $A_{b,\eps} F$ as a smooth function.  
Under the additional restriction $0<b-\frac \ga p<1$ 
the following result extends 
$A_{b,\eps}$ to all of $L^p(\nu_\gamma)$ and  establishes the existence of the limit $A_b=\lim_{\eps\to 0}  A_{b,\eps}$  in the strong operator topology.

\begin{lemma} \label{lem:Acont}
Let  $b \in \R$ with $0 < b-\frac{\ga}{p} < 1$. Then the following holds.

(i) For $\eps>0$, the maps $A_{b,\eps}$ extend to bounded operators \[A_{b,\eps}:  L^p(\nu_\ga)\to \dot W^{b-\frac \ga p,p}, \]   with operator norm uniformly bounded in $\eps$.

(ii) The operators $A_{b,\eps}$ converge to a bounded operator \[A_b:  L^p(\nu_\ga)\to \dot W^{b-\frac \ga p,p},\] in the sense that 
$\lim_{\eps\to 0} \| A_{b,\eps} F-A_b F\|_{\dot W^{b-\frac\ga p,p} }=0$ for all $F\in L^p(\nu_\gamma)$.
\end{lemma}
\begin{proof}
Let $F \in L^p(\nu_{\ga})$ and assume in addition that
assume that   $F\in L^p(\nu_\gamma)$
is compactly supported in $\bbR^d\times(\bbR^d \setminus \{0\} )$. 
Set $$\Delta_{h,h} \vartheta_t(u,v)= \vartheta_t(u+h,v+h)-\vartheta_t(u,v).$$
Then
\[\Delta_h A_{b,\eps}F (x)= 
\int_\eps^{1/\eps} \iint_{\bbR^{2d} } F(y,z)|z|^b \Delta_{h,h}\vartheta_t (x-y-z,x-y) \dif z\dif y\frac{\dif t}{t} \]
and we estimate \[|\Delta_h A_{b,\eps}F(x) | \le I(x,h)+II(x,h)+III(x,h)\] where
\begin{align*}
    I(x,h)&\coloneqq \int_{|h|}^\infty \iint_{\bbR^{2d} }|F(y,z) ||z|^b | \Delta_{h,h}\vartheta_t(x-y-z,x-y)|\dif z\dif y\frac{\dif t}{t} 
    \\
    II(x,h)&\coloneqq \int_0^{|h|} \iint_{\bbR^{2d} }|F(y,z)| |z|^b | \vartheta_t(x+h-y-z,x+h-y)|\dif z\dif y\frac{\dif t}{t} 
    \\
    III(x,h)&\coloneqq \int_0^{|h|} \iint_{\bbR^{2d} }|F(y,z)| |z|^b | \vartheta_t(x-y-z,x-y)|\dif z\dif y\frac{\dif t}{t}. 
\end{align*}
Setting
\Be\label{eq:Jpt} J_p(t)= \Big(\frac{1}{t^d} \int_{\bbR^d}\int_{|z|\le t} |F(y,z)|^p |z|^{bp} \dif z\dif y\Big)^{1/p}
\Ee
we estimate, using Minkowski's inequality,
\[
\begin{split}
\|I(\cdot,&h)\|_p  \leq \int_{|h|}^{\infty} \frac{|h|}{t^2} \Big\| \frac{1}{t^{2d}}\int_{ |x+h-y| \leq 2t} \int_{|z| \leq t} |F(y,z)| |z|^b \dif z \dif y \Big\|_{L^p(\dif x)} \dif t \\
&\leq \int_{|h|}^{\infty} \frac{|h|}{t^2} \Big\| \Big ( \frac{1}{t^{2d}} \int_{ |x+h-y| \leq 2t} \int_{|z| \leq t}  |F(y,z)|^p |z|^{bp} \dif z \dif y \Big)^{1/p} \Big \|_{L^p(\dif x)} \dif t \\&\lesssim \int_{|h|}^{\infty} \frac{|h|}{t^2} J_p(t) \dif t.
\end{split}
\]
Similarly we get
\[
\|II(\cdot,h)\|_p+ \|III(\cdot,h)\|_p  \lesssim \int_0^{|h|} \frac{1}{t} J_p(t)  \dif t.
\]
We then have, uniformly in $\eps\in (0,1)$, 
\[
\begin{split}
\|A_{b,\eps}  &F\|_{\dot{W}^{b-\frac{\ga}{p},p}}=\|\cQ_{1,b} A_{b,\eps} F\|_{L^p(\nu_\ga)} \\&
\leq \Big(\int
\Big[
|h|^{-b} \int_{|h|}^{\infty} \frac{|h|}{t^2} J_p(t)  \dif t +|h|^{-b}  \int_0^{|h|} \frac{1}{t} J_p(t)  \dif t \Big]^p \frac{\dif h}{|h|^{d-\ga}  }\Big)^{1/p} 
\end{split} 
\]
which we estimate (using Hardy's inequalities) by 
\begin{align*} &\Big( \int_0^{\infty} \Big( \int_r^{\infty} \frac{J_p(t)}{t^2} \dif t \Big)^p \frac{\dif r}{r^{1+(b-\frac{\ga}{p}-1)p}} \Big)^{\frac 1p} + \Big( \int_0^{\infty} \Big( \int_0^r \frac{J_p(t)}{t}  \dif t \Big)^p \frac{\dif r}{r^{1+(b-\frac{\ga}{p})p}} \Big)^{\frac 1p} \\
&\lesssim
\Big( \int_0^{\infty} \Big( \frac{J_p(t)}{t^2} \Big)^p t^p \frac{\dif t}{t^{1+(b-\frac{\ga}{p}-1)p}} \Big)^{\frac 1p} + 
\Big( \int_0^{\infty} \left( \frac{J_p(t)}{t} \Big)^p t^p \frac{\dif t}{t^{1+(b-\frac{\ga}{p})p} } \right)^{\frac 1p} 
\\
&\simeq \left( \int_0^{\infty} J_p(t)^p 
\frac{\dif t}{t^{1+(b-\frac{\ga}{p})p}} \right)^{\frac 1p}
\\
& = \left( \int_{\R^d} \int_{\R^d} |F(y,z)|^p |z|^{bp} \int_{|z|}^{\infty} \frac{\dif t}{t^{1+(b-\frac{\ga}{p})p+d}}  \dif z \dif y \right)^{\frac 1p} 
\simeq  \|F\|_{L^p(\nu_{\ga})}.
\end{align*}
This establishes part (i) of the lemma, first for 
$F$ compactly supported in $\bbR^d\times (\bbR^d\setminus \{0\})$ and then, by a density argument,  for general $F\in L^p(\nu_\ga)$.  
The above argument also shows 
that 
$\|A_{b,\eps_1}  F- A_{b,\eps_2} F\|_{\dot{W}^{b-\frac{\ga}{p},p}}\to 0$ as $\eps_1,\eps_2\to 0$  and thus $A_{b,\eps} F$ converges in $\dot W^{b-\frac{\ga}{p},p}$ to a limit $A_{b,0} F$; moreover  $A_{b}$ defines a bounded operator $L^p(\nu_\ga)\to \dot W^{b-\frac{\ga}{p} ,p} $. 
\end{proof}

The proof of Proposition \ref{prop:retract-diff} is now completed by the following lemma.

\begin{lemma} \label{lem:Aretract}
Let  $b \in \R$ with $0 < b-\frac{\ga}{p} < 1$. Then 
$A_b \cQ_{1,b} f = f$,  for  all $f \in \dot{W}^{b-\frac{\ga}{p},p}$.
\end{lemma}
\begin{proof}
Note that 
$\cQ_{1,b} \colon \dot{W}^{b-\frac{\ga}{p},p} \to L^p(\nu_{\ga})$ is an isometry. As $A_b: L^p(\nu_\ga) \to  \dot{W}^{b-\frac{\ga}{p},p}$ is bounded, by Lemma \ref{lem:Acont}, and since $C^{\infty}_c(\R^d)$ is dense in $\dot{W}^{b-\ga/p,p}$, it suffices to prove
$A_b \cQ_{1,b} f = f$,  for  all $f \in C^\infty_c(\bbR^d).$

By \eqref{eq:Phiprop} and  \eqref{eq:vthtdef} we get for each $x\in \bbR^d$
\[
\begin{split}
&A_{b,\eps}  \cQ_{1,b}f(x) = A_{0,\eps}  \cQ_{1,0} f(x)  \\
=& \int_{\varepsilon}^{1/\varepsilon} \int_{\R^d} \int_{\R^d}  (f(y+h)-f(y)) \phi_t(x-y-h) \frac{d}{dt} [ \phi_t(x-y) ]  \dif h \dif y \dif t \\
=&\int_{\varepsilon}^{1/\varepsilon} \int_{\R^d} \int_{\R^d}  (f(z)-f(y)) \phi_t(x-z) \frac{d}{dt} [ \phi_t(x-y) ]  \dif z \dif y \dif t \\
=&\,0- \int_{\varepsilon}^{1/\varepsilon} \int_{\R^d} \int_{\R^d} \int_{\R^d} f(y) \phi_t(x-z) \frac{d}{dt} [ \phi_t(x-y) ]  \dif z \dif y \dif t 
\end{split}
\]
where we used $\int_{\R^d} \frac{d}{dt} [ \phi_t(x-y) ] \dif y = \frac{d}{dt} \int_{\R^d} \phi_t(x-y) \dif y = 0$ to integrate the term involving $f(z)$. We may now integrate in $z$ and $t$ in the last display, using that $\int \phi_t =1$ to obtain for $f\in C^\infty_c(\bbR^d)$
\[
\begin{split}
 A_{b,\eps}  \cQ_{1,b}f(x) &= \int_{\R^d}  f(y) \left( \phi_{\varepsilon}(x-y) - \phi_{1/\varepsilon}(x-y) \right)  \dif y \, .
\end{split}
\]
Letting $\eps\to 0$ we obtain $A_b\cQ_{1,b} f=f$ for $f\in C^\infty_c(\bbR^d). $
\end{proof}



\begin{thebibliography}{99} 

\bibitem{ak-proc}
Irina
Asekritova, Nathan  Kruglyak. 
Real interpolation of vector-valued spaces in non-diagonal case. 
Proc. Amer. Math. Soc. 133 (2005), no. 6, 1665--1675.

\bibitem{ak}
\bysame.
Interpolation of Besov spaces in the nondiagonal case  (Russian).
Algebra i Analiz 18 (2006), no. 4, 1--9; English translation in
St. Petersburg Math. J. 18 (2007), no. 4, 511--516.


\bibitem{ACF} Florent  Autin, Gerda Claeskens,  Jean-Marc Freyermuth. {Hyperbolic wavelet thresholding methods and the curse of dimensionality through the maxiset approach.} Appl. Comput. Harmon. Anal. {36} (2014), 239--255.

 
\bibitem{BL} J\"oran   Bergh, J\"orgen  L\"ofstr\"om. \textit{Interpolation spaces. An introduction.} Grundlehren der Mathematischen Wissenschaften, No. 223. Springer-Verlag, Berlin-New York, 1976.

\bibitem{bergh-peetre} J\"oran Bergh, Jaak  Peetre, 
On the spaces $V^p$ ($0<p\le\infty$). 
Boll. Un. Mat. Ital. (4) 10 (1974), 632--648.

\bibitem{besoy-haroske-triebel} Blanca F. Besoy, Dorothee D. Haroske,  Hans Triebel, 
Traces of some weighted function spaces and related non-standard real interpolation of Besov spaces, preprint	arXiv:2009.03656.

\bibitem{bourdaud-besov} G\'erard Bourdaud.
Realizations of homogeneous Besov and Lizorkin-Triebel spaces. 
Math. Nachr. 286 (2013), no. 5-6, 476--491.


\bibitem{br-sch-yung1} Ha\"{\i}m Brezis, Jean Van Schaftingen,  Po-Lam Yung.  A surprising formula for Sobolev norms.  
Proc.
Natl. Acad. Sci. 118 (2021), no. 8, e2025254118.



\bibitem{br-sch-yung-lor}
\bysame.
Going to Lorentz when fractional Sobolev, Gagliardo and Nirenberg estimates fail. arXiv:2104.09867.  Calc. Var. Partial Differential Equations 60 (2021), no. 4, Paper No. 129, 12 pp.

\bibitem{bsvy} Ha\"{\i}m Brezis, Andreas Seeger, Jean Van Schaftingen,  Po-Lam Yung. Families of functionals representing Sobolev norms.  arXiv:2109.02930, Analysis and PDE, to appear.



\bibitem{christ}
Michael Christ. 
The extension problem for certain function spaces involving fractional orders of differentiability.
Ark. Mat. 22 (1984), no. 1, 63--81.

\bibitem{ciesielski} Zbigniew Ciesielski. Some properties of convex functions of higher orders. Ann. Polon. Math. 7 (1959), 1--7.


 \bibitem{cobos-kruglyak}  Fernando Cobos, Natan Kruglyak.  Exact minimizer for the couple $(L^\infty, BV)$  and the one-dimensional analogue of the Rudin-Osher-Fatemi model. J. Approx. Theory 163 (2011), no. 4, 481--490.

\bibitem{CDD-elliptic} Albert Cohen, Wolfgang  Dahmen, Ronald  DeVore.  Adaptive wavelet methods for elliptic operator equations: convergence rates. Math. Comp. 70 (2001), no. 233, 27--75. 

\bibitem{cohen-et-al} Albert 
Cohen,  Wolfgang Dahmen, Ingrid Daubechies, Ronald DeVore. 
Harmonic analysis of the space $BV$.
Rev. Mat. Iberoamericana 19 (2003), no. 1, 235--263.


\bibitem{daubechies}
Ingrid Daubechies,
\emph{Ten lectures on wavelets.}
CBMS-NSF Regional Conference Series in Applied Mathematics, 61. Society for Industrial and Applied Mathematics (SIAM), Philadelphia, PA, 1992. 

\bibitem{DeVore} Ronald A. DeVore, {Nonlinear approximation}, in Acta Numerica, Vol. 7, Cambridge University Press, pp. 51--150.


\bibitem{DJP} Ronald A. DeVore, Bj\"orn  Jawerth,  Vasil Popov. {Compression of wavelet decompositions}, Amer. J. Math. {114} (1992), 737--785.
	
	
\bibitem{devore-sharpley} 
Ronald A. 
DeVore, Robert C.  Sharpley. \textit{
Maximal functions measuring smoothness.}
Mem. Amer. Math. Soc. 47 (1984), no. 293, viii+115 pp.

 \bibitem{DeVorePetrovaTemlyakov} Ronald A. DeVore, Guergana Petrova,  Vladimir  Temlyakov. {Best basis selection for approximation in $L_p$}. Found. Comput. Math. \textbf{3} (2003), 161--185.

\bibitem{devore-popov}
Ronald A. 
DeVore, 
Vasil A. Popov.
Interpolation of Besov spaces.
Trans. Amer. Math. Soc. 305 (1988), no. 1, 397--414.



\bibitem{djokovic} Dragomir \v Z. Djokovi\'c. A representation theorem for $(X_1-1)\dots(X_n-1)$ and its applications. Ann. Polon. Math. 22 (1969), 189--198.

\bibitem{dominguez-milman} \'Oscar Dom\'inguez, Mario Milman.
New Brezis-Van Schaftingen-Yung Sobolev type inequalities connected with maximal inequalities and one parameter families of operators.
Preprint arXiv 2010.15873.



\bibitem{FJ85} Michael Frazier,  Bj\"orn  Jawerth. {Decomposition of Besov spaces}. Indiana Univ. Math. J. {34} (1985), 777--799.
	
	\bibitem{FJ90} \bysame, {A discrete transform and decompositions of distribution spaces}. J. Funct. Anal. {93} (1990), 34--170.


\bibitem{freitag} Dietrich  Freitag. Real interpolation of weighted $L^p$-spaces. Math. Nachr. 86 (1978), 15--18. 





\bibitem{ger} Roman Ger. On almost polynomial functions, Colloq. Math. 24 (1971/72), 95--101.

\bibitem{greco-schiattarella} Luigi Greco, Roberta Schiattarella. 
An embedding theorem for $BV$-functions. 
Commun. Contemp. Math. 22 (2020), Article ID 1950032, 13 p. 


 \bibitem{Gribonval} R\'emi  Gribonval,  Morten  Nielsen. {Some remarks on non-linear approximation with Schauder bases}. East J. Approx. {7} (2001), 1--19.



\bibitem{gu-yung} Qingsong Gu, Po-Lam Yung. A new formula for the $L^p$  norm. 
J. Funct. Anal. 281 (2021), no. 4, 109075.



\bibitem{hansen} Markus Hansen. Nonlinear approximation rates and Besov regularity for elliptic PDEs on polyhedral domains. 
Found. Comput. Math. 15 (2015), no. 2, 561--589.

\bibitem{HS} Markus  Hansen,  Winfried  Sickel. {Best $m$-term approximation and Triebel--Lizorkin spaces}. J. Approx. Theory {163} (2011), 923--954.

\bibitem{hunt}
Richard A. 
Hunt. 
On $L(p,q)$ spaces.
Enseign. Math. (2) 12 (1966), 249?276.

\bibitem{iwaniec-martin-sbordone} Tadeusz 
Iwaniec, Gaven  Martin, Carlo  Sbordone. 
$L^p$-integrability \& weak type $L^2$-estimates for the gradient of harmonic mappings of $\bbD$.
Discrete Contin. Dyn. Syst. Ser. B 11 (2009), no. 1, 145--152.


		
 	\bibitem{Kerkyacharian} Gerard Kerkyacharian,  Dominique  Picard. {Entropy, universal coding, approximation, and bases properties}. Constr. Approx. {20} (2004), 1--37.

\bibitem{krepkogorskii}
V.L. Krepkogorski\u \i. 
Interpolation in Lizorkin-Triebel and Besov spaces. (Russian). 
Mat. Sb. 185 (1994), no. 7, 63--76; translation in
Russian Acad. Sci. Sb. Math. 82 (1995), no. 2, 315--326.

\bibitem{krepkogorskii2} \bysame.
Interpolation and embedding theorems for quasinormed Besov spaces. (Russian)
Izv. Vyssh. Uchebn. Zaved. Mat. 1999, no. 7, 23--29; translation in
Russian Math. (Iz. VUZ) 43 (1999), no. 7, 21--26

\bibitem{kuczma}
Marek Kuczma.
On measurable functions with vanishing differences.
Ann. Math. Sil. No. 6 (1992), 42?60.


\bibitem{kuczma-book}
\bysame.
\textit{An introduction to the theory of functional equations and inequalities. Cauchy's equation and Jensen's inequality.} Second edition. 
Edited and with a preface by Attila Gil\'anyi.  xiv+595 pp.
Birkh\"auser Verlag, Basel, 2009. xiv+595 pp. 

\bibitem{kyriazis}
G. Kyriazis.
Decomposition systems for function spaces. 
Studia Math. 157 (2003), no. 2, 133--169.

\bibitem{lrs-stein} Sanghyuk Lee, Keith M. Rogers, Andreas Seeger.  Square functions and maximal operators associated with radial Fourier multipliers. \textit{Advances in analysis: the legacy of Elias M. Stein}.  273--302, Princeton Math. Ser., 50, Princeton Univ. Press, Princeton, NJ, 2014. 

\bibitem{lizorkin}P. I. Lizorkin. Interpolation of $L^p$-spaces with a weight. Trudy Mat. Inst. Steklov. 140 (1976), 201--211 (in Russian); English transl.: Proc. Steklov Inst. Math. 140 (1979), 221--232. 

 \bibitem{mirek-stein-zorin} Mariusz
 Mirek, Elias M. Stein, Pavel  Zorin-Kranich.
 Jump inequalities via real interpolation. 
 Math. Ann. 376 (2020), no. 1-2, 797--819.

\bibitem{madani} Madani Moussai. Characterizations of realized homogeneous Besov and
Triebel-Lizorkin spaces via differences. Appl. Math. J. Chinese Univ.
2018, 33(2): 188--208.



\bibitem{nikolskii} S.M. Nikol'ski\u{\i}.
\textit{Approximation of functions of several variables and imbedding theorems.}
Translated from the Russian by John M. Danskin, Jr. Die Grundlehren der mathematischen Wissenschaften, Band 205. Springer-Verlag, New York-Heidelberg. 1975.

\bibitem{nguyen06} Hoai-Minh Nguyen. Some new characterizations of Sobolev spaces.
J. Funct. Anal., 237 (2006), no. 2, 689--720.


\bibitem{Petrushev} Pencho P. Petrushev,  Vasil  A. Popov.  Rational Approximation of Real Functions, in: Encyclopedia of Mathematics and its Applications, vol. 28, Cambridge University Press, Cambridge, 1987.


\bibitem{poliakovsky}
Arkady Poliakovsky.
Some remarks on a formula for Sobolev norms due to Brezis, Van Schaftingen and Yung. J. Funct. Anal., 282 (2022), no. 3, 109312. 

\bibitem{Rodiac_VanSchaftingen_2021} 
R\'emy Rodiac, Jean Van Schaftingen. 
Metric characterization of the sum of fractional Sobolev spaces. {Studia Math.},
{258},
(2021),
 {27--51}.

\bibitem{Rivoirard} Vincent  Rivoirard. {Nonlinear estimation over weak Besov spaces and minimax Bayes method}. Bernoulli {12} (2006), 609--632.
	


\bibitem{se-TL} Andreas Seeger.
A note on Triebel-Lizorkin spaces. \textit{Approximation and function spaces} (Warsaw, 1986), 391--400, Banach Center Publ., 22, PWN, Warsaw, 1989.


\bibitem{SeTr} Andreas Seeger, Walter Trebels.  Embeddings for spaces of Lorentz-Sobolev type. Math. Ann. 373 (2019), no. 3-4, 1017--1056. 

\bibitem{stein-diff} Elias M. Stein. \textit{Singular integrals and differentiability properties of
              functions.} Princeton Mathematical Series, No. 30,
Princeton University Press, Princeton, N.J., 1970.

\bibitem{stein-weiss} Elias M. Stein, Guido  Weiss. 
\textit{Introduction to Fourier analysis on Euclidean spaces.}
Princeton Mathematical Series, No. 32. Princeton University Press, Princeton, N.J., 1971. x+297 pp.


\bibitem{taibleson} Mitchell Taibleson.
On the theory of Lipschitz spaces of distributions on Euclidean $n$-space. I. Principal properties.
J. Math. Mech. 13 (1964) 407--479.

 \bibitem{triebel} Hans  Triebel.
\textit{Theory of function spaces.}
Monographs in Mathematics, 78. Birkh\"auser Verlag, Basel, 1983. 284 pp.

 \bibitem{triebel2} \bysame. 
\textit{Theory of function spaces. II.}
Monographs in Mathematics, 84. Birkh\"auser Verlag, Basel, 1992. viii+370 pp.


\end{thebibliography}
\end{document}